\documentclass{elsarticle} 

\usepackage{amsmath,amssymb,amsfonts,graphicx,hyperref}
\usepackage[margin=1in]{geometry}

\usepackage{amstext}
\usepackage{stmaryrd}
\usepackage{algpseudocode}

\usepackage[normalem]{ulem}
\usepackage{graphicx}

\usepackage{array}
\usepackage{verbatim}
\usepackage{booktabs}
\usepackage{calc}
\usepackage{textcomp}
\usepackage{multirow}

\usepackage[all]{xy}

\usepackage{cancel}
\usepackage{mathtools}
\usepackage{placeins}
\usepackage{xspace}

\usepackage{tikz,tikzscale}
\usetikzlibrary{calc}
\usetikzlibrary{shadings}

\usepackage[mode=multiuser,status=draft]{fixme} 
\fxsetup{envlayout=color}
\fxsetup{innerlayout=inline}
\fxsetup{targetlayout=colorcb}
\FXProvidesTargetLayout{color}

\fxusetheme{color}
\FXRegisterAuthor{ss}{envss}{SS}
\FXRegisterAuthor{mr}{envmr}{MR}
\FXRegisterAuthor{mw}{envmw}{MW}
\FXRegisterAuthor{mk}{envmk}{MK}

\usepackage{subcaption}
\usepackage[format=plain,indention=.5cm]{caption}
\usepackage{pgfplots}
\pgfplotsset{compat=1.18}
\usepgfplotslibrary{fillbetween}

\usepackage{bm,upgreek}

\newcommand{\Qcont}[1]{\mathbb{Q}_{#1}}
\newcommand{\Qdg}[1]{\mathbb{Q}^{\mathrm{DG}}_{#1}}
\newcommand{\pairDGthree}{\mathbb{Q}_3/\mathbb{Q}^{\mathrm{DG}}_1}
\newcommand{\pairDGtwo}{\mathbb{Q}_2/\mathbb{Q}^{\mathrm{DG}}_0}
\usepackage{graphicx} 
\usepackage[linesnumbered,vlined,commentsnumbered,ruled]{algorithm2e}

\definecolor{darkgreen}{rgb}{0.0, 0.5, 0.0}

\providecommand{\pr}{^{\prime}}

\usepackage{amsthm}
\newtheorem{theorem}{Theorem}
\newtheorem{lemma}{Lemma}

\newtheorem{corollary}{Corollary}
\theoremstyle{definition}
\newtheorem{assumption}{Assumption}
\theoremstyle{remark}
\newtheorem{remark}{Remark}

\makeatletter
\def\ps@pprintTitle{%
 \let\@oddhead\@empty
 \let\@evenhead\@empty
 \def\@oddfoot{\reset@font\hfil\thepage\hfil}%
 \let\@evenfoot\@oddfoot}
\makeatother

\title{Block Preconditioning for Shifted Boundary Method Discretisations of the Stokes Problem}

\author[heid]{Micha\l{} Wichrowski\corref{cor}}
\author[heid]{Ajay Ajith}

\cortext[cor]{Corresponding author. \texttt{mwichro@mimuw.edu.pl}}
\address[heid]{Interdisciplinary Center for Scientific Computing, Heidelberg University, Germany}

\begin{document}

\begin{abstract}
    The Shifted Boundary Method (SBM) sidesteps body-fitted meshing by shifting boundary conditions onto a surrogate
boundary and correcting for the displacement through Taylor expansions. Despite its broad analysis and application,
scalable iterative solvers for the incompressible Stokes equations remain underdeveloped. We present a block
preconditioner for SBM--Stokes discretisations that uses the velocity block together with a pressure mass matrix as a
Schur complement approximation. Because the SBM system is non-symmetric, classical operator preconditioning does not
apply directly; a field-of-values analysis instead shows that the non-symmetric SBM contributions act as asymptotically
small perturbations of a standard saddle-point operator, yielding mesh-independent GMRES convergence on sufficiently
fine meshes. Numerical experiments demonstrate iteration counts under refinement across geometries of increasing
complexity. We expose a coarse-mesh regime in which an under-resolved grid produces elevated iteration counts, an
artefact of insufficient resolution that vanishes once the mesh captures the geometry.

\end{abstract}

\maketitle

\section{Introduction}
\label{sec:intro}

Capturing the intricate detail of real-world geometries within a computational grid is a persistent challenge in the
simulation of incompressible flow. The traditional finite element route demands a body-fitted mesh whose cells conform
to the domain boundary. While effective, this bespoke mesh generation is computationally expensive and frequently
becomes the bottleneck of the entire pipeline for intricate three-dimensional shapes~\cite{main2018shifted}, requiring
manual intervention to keep elements near curved or narrow boundaries from degenerating into ill-shaped slivers.
Unfitted finite element methods offer an alternative by employing a fixed background mesh that remains independent of
the physical boundary, trading bespoke meshing for a more complex algebraic structure.

Among these unfitted approaches, the Shifted Boundary Method (SBM)~\cite{main2018shifted} shifts the location where
boundary conditions are applied to meet the mesh. By defining the problem on a surrogate domain and extrapolating
boundary conditions onto it (typically via Taylor expansions or more general extension
operators~\cite{zorrilla2024shifted}, enforced in a Nitsche-like manner~\cite{nitsche1971variationsprinzip}), SBM
avoids the complex geometric intersections of methods such as CutFEM~\cite{burman2015cutfem} and the associated
cut-cell quadrature. However, SBM shifts part of the computational challenge from the mesh generator to the linear
solver: the extrapolation terms introduce non-symmetry and potential indefiniteness, so a simple mesh no longer
guarantees a simple matrix. While the conditioning of SBM scales like $\mathcal{O}(h^{-2})$~\cite{atallah2021shifted},
comparably to body-fitted methods~\cite{collins2023penalty}, the efficient solution of the algebraic systems arising
from high-order SBM formulations remains largely unexplored.

The SBM has matured rapidly since its first formulation~\cite{main2018shifted, main2018shiftedVol2}. The original
construction used cells strictly within the domain, whereas more recent approaches~\cite{yang2024optimal} also admit
intersected cells based on a volume-fraction threshold. A complete stability and convergence theory now covers the
Stokes problem~\cite{atallah2020analysis}, the Poisson problem on domains with corners~\cite{atallah2021analysis}, and
high-order discretisations of arbitrary polynomial degree~\cite{atallah2022high}. The method has been extended to
Isogeometric Analysis~\cite{ANTONELLI2024shiftedIGA}, to a broad range of physics including solid
mechanics~\cite{atallah2021shifted, atallah2024nonlinear}, and to problems with embedded
interfaces~\cite{li2020shifted, xu2024weighted}, where the formulation is modified to impose jump conditions across
internal boundaries. Further innovations include penalty-free variants~\cite{collins2023penalty} and integration with
level-set methods~\cite{kuzmin2022unfitted, xue2021new}, which facilitate operations such as the closest-point
projection that SBM relies on to relate the surrogate boundary $\tilde{\Gamma}$ to the true boundary $\Gamma$.

The main geometric task in SBM is to determine accurately the relationship between points on $\tilde{\Gamma}$ and
$\Gamma$, which avoids integration over the arbitrarily shaped domains that arise from cell-boundary intersections. In
the broader unfitted landscape, CutFEM~\cite{burman2015cutfem} discretises directly on the physical domain by cutting
background cells, ensuring conditioning through geometry-adapted quadrature and stabilisation such as the ghost
penalty~\cite{burman2010ghost}. This comes at the cost of specialised quadrature in every term of the bilinear form on
every cut cell, an overhead incurred at each operator evaluation in a matrix-free setting. SBM, by contrast, retains
the efficiency of standard tensor-product quadrature, confining the geometric work to a one-time preprocessing step.
The algebraic challenges of SBM are therefore worth confronting directly.

The geometric flexibility of SBM introduces algebraic challenges. The extrapolation of boundary conditions yields
linear systems that are non-symmetric and potentially indefinite~\cite{wichrowski2025geometric}, particularly for
higher-order polynomial approximations, demanding specialised preconditioning. Multigrid
methods~\cite{hackbusch1985multi} are the natural candidates for the large systems arising from such discretisations,
yet their use as preconditioners for SBM remains largely unexplored. Algebraic Multigrid (AMG) has been applied to SBM
for continuous linear elements~\cite{atallah2020second}, but it generally struggles with high-order discretisations
even on body-fitted meshes, and is less effective still for unfitted methods with $p > 1$; it also fails to exploit the
geometric regularity of structured background grids. Recent work has addressed this through geometric multigrid with
subspace-correction patch smoothers tailored to the SBM boundary coupling~\cite{wichrowski2025geometric, shypatch},
with efficient matrix-free realisations in view~\cite{wichrowski2025local, wichrowski2025matrix}.

Every one of these solver advances targets \emph{symmetric positive-definite} (SPD) problems for which multigrid
smoothers and subspace corrections are well understood. The Stokes problem lacks this structure entirely. The
incompressibility constraint $\nabla \cdot \mathbf{u} = 0$ produces a saddle-point system with a zero diagonal block,
rendering it indefinite, so the solvers developed for SBM in~\cite{wichrowski2025geometric, shypatch} do not apply
without substantial modification.

For body-fitted Stokes discretisations, a well-established theory governs the solution of the resulting saddle-point
systems. Two broad families dominate. The first treats the system through \emph{block preconditioners} acting on the
$2\times 2$ structure: operator-compatible preconditioners~\cite{PreconditioningPDEs} exploit the spectral equivalence
of the pressure mass matrix with the Schur complement $S = B A^{-1} B^{T}$ to deliver mesh-independent convergence of
Krylov solvers such as MinRes and GMRES~\cite{BlockPreconditioners, BlockPrec2, Elman2014-gj}, provided the
discretisation is \textit{inf-sup} stable. The convergence of these block-diagonal and block-triangular preconditioners
is characterised through eigenvalue and field-of-values bounds~\cite{murphy2000note, loghin2004analysis,
    klawonn1999fov, starke1997fov, eisenstat1983variational}, the field-of-values arguments being particularly relevant for
the non-symmetric, GMRES-driven setting considered here. The second family applies multigrid directly to the coupled
system, relying on smoothers that respect the saddle-point structure: the Braess--Sarazin
smoother~\cite{braess1997efficient}, the closely related inexact smoothers of Zulehner~\cite{zulehner2000class}, and
the local Vanka-type relaxations~\cite{vanka1986block} that solve small coupled subproblems cell by cell. A unified
analysis of such smoothers and of multigrid for indefinite Stokes systems was given by Schöberl and
Zulehner~\cite{schoberl2003schwarz}. One successful approach within this family combines \emph{divergence-conforming}
discretisations ($H(\mathrm{div})$ elements such as Raviart--Thomas or BDM, which yield pointwise divergence-free
velocities) with overlapping Schwarz (vertex-patch) smoothers, for which multigrid convergence has been established by
Arnold, Falk, and Winther~\cite{arnold2000multigrid} and developed into highly efficient, $p$- and Reynolds-robust
solvers by Kanschat, Mao, and Sch\"oberl and co-workers~\cite{kanschat2015multigrid, farrell2019augmented}. A
complementary line of work builds \emph{constrained} smoothers that enforce the divergence constraint within the
relaxation; while not central to the present block-preconditioning approach, these strategies (developed for Stokes and
related constrained problems~\cite{wichrowski2022stokes, jodlbauer2024matrix}) illustrate the alternative of building
the saddle-point structure into the smoother rather than into an outer block preconditioner.

For \emph{unfitted} methods the picture is far less settled. Tailored preconditioners have been demonstrated for
immersed isogeometric Stokes~\cite{deprenter2019preconditioning} and for the Stokes--immersed-boundary saddle
point~\cite{zhang2014ibm, bano2025ibm}, while optimal subspace decompositions have been analysed for CutFEM elliptic
problems, with the CutFEM Stokes case left explicitly open~\cite{gross2023cutfem}. Each construction, however, is bound
to the algebraic structure of its own immersed method. The SBM extrapolation modifies the velocity block $A$ and the
divergence block $B$ in ways distinct from cut-cell and immersed-boundary methods, and to the best of our knowledge no
block preconditioning strategy has been analysed for the saddle-point systems arising from \emph{SBM discretisations of
    the Stokes equations}.

This work addresses this problem. We show that high-order SBM--Stokes systems can be solved with mesh-independent
efficiency using a block upper-triangular preconditioner built from the velocity block and a pressure-mass-matrix Schur
complement approximation~\cite{PreconditioningPDEs, Elman2014-gj}. Numerical experiments investigate Krylov iteration
counts for the \textit{inf-sup} stable $\pairDGthree$ and $\pairDGtwo$ element pairs across a range of embedded
geometries. We additionally investigate the minimum resolution required to resolve surrogate-boundary features: on
coarse grids, where surrogate evaluations may conflict, iteration counts degrade, but the ideal mesh-independent
performance is recovered once the grid resolves the geometry. The implementation is built on the \texttt{deal.II}
finite element library~\cite{dealii2019design, dealii2024}.

The remainder of the paper is organised as follows. Section~\ref{sec:methods} formulates the Shifted Boundary Method
for the Poisson and Stokes problems, establishing the variational framework, the choice of finite element spaces, and
the block preconditioner. Section~\ref{sec:implementation} describes the implementation within \texttt{deal.II},
focusing on the assembly of the SBM terms and the construction of the block preconditioner. Section~\ref{sec:results}
presents numerical experiments validating the mesh-independence of the solver for the $\pairDGthree$, $\pairDGtwo$, and
Taylor--Hood discretisations across different geometric arrangements. Section~\ref{sec:conclusion} provides concluding
remarks.

\section{Method formulation}

\label{sec:methods}

We first introduce the Shifted Boundary Method (SBM) for the Poisson equation, which provides a simple setting in which
to present its key ingredients before extending them to the Stokes equations. We follow
\cite{shypatch,wichrowski2025geometric} for the Poisson problem and \cite{atallah2020analysis} for the Stokes problem.

\subsection{Poisson equation}

Consider the Poisson problem
\begin{align}
    -\Delta u & = f \quad \text{in } \Omega, \label{eq:poisson}        \\
    u         & = u_d \quad \text{on } \Gamma, \label{eq:dirichlet_bc}
\end{align}
where $\Omega \subset \mathbb{R}^d$ is a domain with boundary $\Gamma = \partial\Omega$, $f$ is a given source term, and $u_d$ is the prescribed Dirichlet datum on $\Gamma$.

Seeking the solution $u \in H^1(\Omega)$, multiplying \eqref{eq:poisson} by a test function $v$, and integrating by
parts yields
\begin{equation}
    \int_\Omega \nabla u \cdot \nabla v \, dx  - \int_{\Gamma} (\nabla u \cdot \mathbf{n})\, v \, ds = \int_\Omega f v \, dx ,
\end{equation}
where $\mathbf{n}$ denotes the outward unit normal to $\Gamma$.

We next introduce a triangulation \( \mathcal{T}_h \) consisting of quadrilateral (2D) or hexahedral (3D) elements of
size $h$, and define a finite element space $V_h\subset H^1(\Omega)$ of continuous Lagrange elements of degree \( p \).
In classical finite element methods the mesh conforms to the boundary (unlike the background-mesh approach illustrated
in Figure~\ref{fig:sbm_setup}), and the Dirichlet condition is enforced strongly by constraining the degrees of freedom
on $\Gamma$. When the function space does not satisfy the essential boundary conditions by construction, they must
instead be enforced weakly through additional integral terms on $\Gamma$.

Nitsche's method \cite{nitsche1971variationsprinzip} imposes Dirichlet boundary conditions weakly by augmenting the
bilinear form with consistency and penalty terms on $\Gamma$. The standard Nitsche formulation for the Poisson problem
seeks $u_h \in V_h$ such that for all $v_h \in V_h$:
\begin{equation}
    \begin{split}
        \int_\Omega \nabla u_h \cdot \nabla v_h \, dx
        \;\; - \int_{\Gamma} (\nabla u_h \cdot \mathbf{n}) v_h \, ds
        - & \sigma\int_{\Gamma} (\nabla v_h \cdot \mathbf{n}) u_h \, ds
        + \int_{\Gamma} \alpha u_h v_h \, ds                            \\
          & = \int_\Omega f v_h \, dx
        - \sigma\int_{\Gamma} (\nabla v_h \cdot \mathbf{n}) u_d \, ds
        + \int_{\Gamma} \alpha u_d v_h \, ds.
    \end{split}
    \label{eq:nitsche_poisson}
\end{equation}
Here, $\alpha$ is a penalty parameter, typically chosen as $\alpha = \mathcal{O}(p^2 h^{-1})$. The choice $\sigma=1$
yields a symmetric formulation, while $\sigma=-1$ results in a non-symmetric formulation in which the penalty term can
be omitted~\cite{baumann1999discontinuous}.

\subsection{Shifted boundary method}

The SBM is an approximate-domain method for boundary value problems: the location at which boundary conditions are
imposed is shifted from the true boundary to a nearby surrogate boundary made up of mesh faces.

The boundary data are defined on $\Gamma$, whereas the discrete solution lives only on the surrogate domain $\Omega\pr$
with boundary $\Gamma\pr$ (see Figure~\ref{fig:sbm_setup}). To reconcile the two, the solution is extended from
$\Omega\pr$ towards $\Gamma$ by a Taylor expansion of the same order as the finite element basis functions \cite{SBMP1,
    atallah2022high}. For every point on the surrogate boundary $\Gamma\pr$, we denote by $u_p$ the Taylor extension of $u$
evaluated at the corresponding projected point on the true boundary $\Gamma$. The boundary conditions are then imposed
via Nitsche's method, comparing $u_p$ with the boundary datum $u_d$.

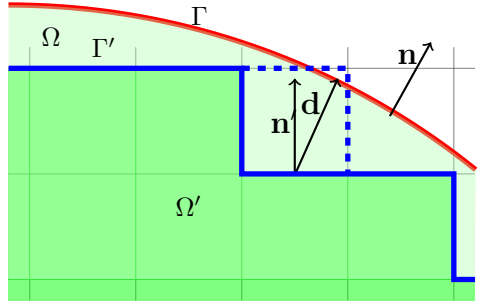
\begin{figure}[!ht]
    \centering
    \centering
\begin{tikzpicture}[scale=1.4]
    \draw[step=1cm, gray, very thin] (0.8,-0.2) grid (5.2,2.2);

    \foreach \x/\y in {2/1, 1/0, 1/1, 2/0, 3/0, 4/0}
        {
            \fill[green, opacity=0.5] (\x,\y) rectangle (\x+1,\y+1);
        }
    \foreach \x/\y in {1/0 }{
            \fill[green, opacity=0.5] (\x,\y) rectangle (\x-0.2,\y+2);
        }

    \fill[green, opacity=0.5] (0.8,-0.2) rectangle (5.2,0);

    \draw[line width=2pt, red, domain=51:90, samples=100, variable=\t]
    plot ({0.8 + 7.0*cos(\t)}, {-4.4 + 7.0*sin(\t)});
    \node at (2.6,2.5) {$\Gamma$};

    \fill[green!30, opacity=0.4] plot[domain=51:90, samples=100, variable=\t] ({0.8 + 7.0*cos(\t)}, {-4.4 +
            7.0*sin(\t)}) -- (0.8,0) -- (5.2,0) -- cycle;

    \pgfmathsetmacro{\myangle}{66}
    \draw[<-, thick] ({0.8 + 6.9*cos(\myangle) + 0.3}, {-4.4 + 6.9*sin(\myangle)}) -- ({0.8 + 5.9*cos(\myangle)
            +
            0.3}, {-4.4 +
            5.9*sin(\myangle)});
    \pgfmathsetmacro{\labelangle}{\myangle - 6} 
    \node at ({0.8 + 6.9*cos(\labelangle) - 0.6}, {-4.4 + 6.9*sin(\labelangle)+0.1}) {\large $\mathbf{d}$};

    \draw[->, thick] ({0.8 + 6.8*cos(\myangle-5) + 0.3}, {-4.4 + 6.8*sin(\myangle-5)}) -- ({0.8 +
            7.6*cos(\myangle-5)
            +
            0.3}, {-4.4 +
            7.6*sin(\myangle-5)});
    \node at ({0.8 + 7.5*cos(\labelangle-5) - 0.54}, {-4.4 + 7.5*sin(\labelangle-5)+0.36}) {\large $\mathbf{n}$};

    \draw[<-, thick]  ({0.8 + 5.9*cos(\myangle) +
            0.3}, {-4.4 + 6.9*sin(\myangle)}) -- ({0.8 + 5.9*cos(\myangle) +
            0.3}, {-4.4 +
            5.9*sin(\myangle)});
    \node at (3.4, 1.5) {\large $\mathbf{n}\pr$}; 

    \draw[line width=2pt, blue] (0.8,2) -- (3,2) -- (3,1) -- (4,1) -- (5,1)-- (5,0)-- (5.2,0.) ;

    \draw[line width=2pt, blue, dashed] (3,2) -- (4,2) -- (4,2) -- (4,1) ;
    \node at (1.7,2.2) {$\Gamma\pr$};

    \node at (2.5,0.7) {$\Omega\pr$};

    \node at (1.2,2.3) {${\Omega}$};

\end{tikzpicture}
    \caption{Schematic illustrating the background mesh, interior cells (green), the surrogate boundary
        $\Gamma\pr$ (thick blue line) along the upper boundary of the interior cells, and the true boundary
        $\Gamma = \partial\Omega$ (red). Figure adapted from \cite{shypatch,wichrowski2025geometric}.}
    \label{fig:sbm_setup}
\end{figure}

Applying the symmetric Nitsche formulation \eqref{eq:nitsche_poisson} (with $\sigma=1$) on the surrogate boundary, with
the extended values $u_p$ in place of the trace of $u_h$, yields the SBM formulation of the Poisson problem:
\begin{equation}
    \begin{split}
        \int_{\Omega \pr} \nabla u \cdot \nabla v \, dx - \int_{\Gamma \pr} \partial_{n \pr}u \, v \, ds
        - \int_{\Gamma \pr} u_p \, \partial_{n \pr}v \, ds
        + \alpha \int_{\Gamma \pr} u_p v \, ds
        \\
        = \int_{\Omega \pr} f v \, dx
        - \int_{\Gamma \pr} u_d \, \partial_{n \pr}v \, ds
        + \alpha \int_{\Gamma \pr} u_d v \, ds,
    \end{split}
\end{equation}
where $\partial_{n\pr}$ denotes the derivative in the direction of the surrogate normal $\mathbf{n}\pr$.

\subsection{SBM for the Stokes equations}
\label{sec:sbm_stokes}
The Stokes problem considered in this work reads: find the velocity $\mathbf{u}$ and the pressure $p$ such that
\begin{align}
    -\nabla \cdot \left(2\boldsymbol{\nabla}_s \mathbf{u}\right) + \nabla p & = \mathbf{f}   & \text{in } \Omega ,     \label{eq:stokes_1}    \\
    -\nabla \cdot \mathbf{u}                                                & = 0            & \text{in } \Omega,      \label{eq:stokes_2}    \\
    \mathbf{u}                                                              & = \mathbf{u}_d & \text{on } \Gamma_D,    \label{eq:stokes_bc_d} \\
    \left(2\boldsymbol{\nabla}_s \mathbf{u} - p\mathbf{I}\right)\mathbf{n}  & = \mathbf{0}   & \text{on } \Gamma_N,    \label{eq:stokes_bc_n}
\end{align}
where the boundary splits into a Dirichlet part $\Gamma_D$ and a Neumann part $\Gamma_N$ with $\Gamma = \overline{\Gamma_D \cup \Gamma_N}$, and $\boldsymbol{\nabla}_s\mathbf{u}$ denotes either the symmetric gradient $\varepsilon(\mathbf{u}) = \tfrac12\left(\nabla\mathbf{u} + \nabla\mathbf{u}^T\right)$ or the regular gradient $\tfrac12\nabla\mathbf{u}$; both choices are carried through the analysis below (see \eqref{eq:stress_tensor}). Here $\mathbf{f}$ is a given body force, $\mathbf{u}_d$ is the prescribed Dirichlet datum on $\Gamma_D$, and \eqref{eq:stokes_bc_n} is the homogeneous ``do-nothing'' natural condition imposed on $\Gamma_N$.

We adopt a high-order mixed discretisation that pairs a continuous velocity space with a discontinuous pressure space.
The velocity space $V_h \subset [H^1(\Omega)]^d$ consists of continuous Lagrange elements of degree $k=3$, with local
space
\[
    V_h(K) = [Q_3(K)]^d
\]
on each element $K$, where $Q_3$ denotes polynomials of degree at most~$3$ in each coordinate direction. The pressure
space $Q_h \subset L^2(\Omega)$ consists of discontinuous elements of degree $k=1$, with local space
\[
    Q_h(K) = P_1(K),
\]
where $P_1$ denotes the space of linear polynomials; the discontinuous pressure space allows jumps across element
interfaces. This pair, denoted $\pairDGthree$, is inf-sup stable \cite{atallah2020analysis}. We also consider the
$\pairDGtwo$ pair, with velocity of degree $k=2$ and piecewise-constant pressure, which is likewise inf-sup stable
\cite{atallah2020analysis}.

To simplify the variational formulation, we introduce the combined stress tensor
\begin{equation}
    \boldsymbol{\sigma}(\mathbf{u},p) = 2 \boldsymbol{\nabla}_s\mathbf{u} - p \mathbf{I} ,
    \qquad
    \boldsymbol{\nabla}_s\mathbf{u} \in \{\, \varepsilon(\mathbf{u}),\; \tfrac12\nabla\mathbf{u} \,\} ,
    \label{eq:stress_tensor}
\end{equation}
where the operator $\boldsymbol{\nabla}_s$ may be taken either as the symmetric gradient
$\varepsilon(\mathbf{u}) = \tfrac12(\nabla\mathbf{u}+\nabla\mathbf{u}^T)$, recovering the
physical stress $\boldsymbol{\sigma} = 2\varepsilon(\mathbf{u}) - p\mathbf{I}$, or as
$\tfrac12\nabla\mathbf{u}$, giving the regular-gradient (pseudo-stress) form
$\boldsymbol{\sigma} = \nabla\mathbf{u} - p\mathbf{I}$. For the divergence-free velocity field
of \eqref{eq:stokes_2} both choices yield the same strong-form momentum balance, the vector
Laplacian $-\mu\Delta\mathbf{u} + \nabla p$, and differ only in the natural (traction) boundary
condition \eqref{eq:stokes_bc_n}. The two formulations coincide as discretisations on
the Dirichlet boundary $\Gamma_D$, where the trace of $\mathbf{u}$ is prescribed and enters the
weak form only through the data $\mathbf{u}_d$; they differ on $\Gamma_N$, where each leaves its
own do-nothing condition: $(2\varepsilon(\mathbf{u}) - p\mathbf{I})\mathbf{n} = \mathbf{0}$ for
the symmetric gradient and $(\nabla\mathbf{u} - p\mathbf{I})\mathbf{n} = \mathbf{0}$ for the
regular gradient. The do-nothing condition is enforced \emph{weakly and consistently}
in both cases by omitting the boundary integral on $\Gamma_N$ in the weak form below, so no
explicit traction term distinguishes the two formulations in the assembled system. The SBM and
Nitsche machinery, and the entire stability analysis of Section~\ref{sec:theory}, act only on the
embedded Dirichlet boundary $\Gamma_D$ (and its surrogate $\Gamma\pr$), where the two choices
agree; the analysis applies verbatim to both. We carry $\boldsymbol{\nabla}_s$
through the derivation below; unless stated otherwise, identities written with
$\varepsilon(\cdot)$ hold verbatim for $\tfrac12\nabla(\cdot)$.

Multiplying \eqref{eq:stokes_1} and \eqref{eq:stokes_2} by test functions $\mathbf{v}$ and $q$, respectively, summing,
integrating over the domain, and integrating the stress term by parts, we obtain
\begin{equation}
    \int_\Omega \nabla \mathbf{v} : \boldsymbol{\sigma}(\mathbf{u},p) \, dx - \int_{\partial\Omega} \mathbf{n} \cdot \boldsymbol{\sigma}(\mathbf{u},p)\cdot\mathbf{v} \, ds - \int_\Omega q \, \nabla \cdot \mathbf{u} \, dx = \int_\Omega \mathbf{v} \cdot \mathbf{f} \, dx .
\end{equation}
On $\Gamma_N$ the boundary integrand $\mathbf{n}\cdot\boldsymbol{\sigma}(\mathbf{u},p)\cdot\mathbf{v}$ vanishes
by the do-nothing condition \eqref{eq:stokes_bc_n} and is dropped; only the contribution on $\Gamma_D$
remains. To impose the Dirichlet condition $\mathbf{u} = \mathbf{u}_d$ weakly there, we apply Nitsche's method,
adding symmetric consistency and penalty terms on the Dirichlet boundary $\Gamma_D$:
\begin{equation}
    \begin{split}
        \int_\Omega \nabla \mathbf{v} : \boldsymbol{\sigma}(\mathbf{u},p) \, dx - \int_\Omega q \, \nabla \cdot \mathbf{u} \, dx
        - \int_{\Gamma_D} \mathbf{n} \cdot \boldsymbol{\sigma}(\mathbf{u},p)\cdot\mathbf{v} \, ds                                                            \\
        - \int_{\Gamma_D} \mathbf{n} \cdot \boldsymbol{\sigma}(\mathbf{v},q)\cdot\mathbf{u} \, ds + \alpha \int_{\Gamma_D} \mathbf{u} \cdot \mathbf{v} \, ds \\
        = \int_\Omega \mathbf{v} \cdot \mathbf{f} \, dx
        - \int_{\Gamma_D} \mathbf{n} \cdot \boldsymbol{\sigma}(\mathbf{v},q)\cdot\mathbf{u}_d \, ds + \alpha \int_{\Gamma_D} \mathbf{u}_d \cdot \mathbf{v} \, ds .
    \end{split}
\end{equation}
Finally, we apply the SBM shift: the volume integrals are evaluated on the surrogate domain $\Omega\pr$, and the
boundary integrals on the surrogate boundary $\Gamma\pr$ with its normal $\mathbf{n}\pr$. To impose the boundary
conditions consistently from the surrogate boundary, the traces of $\mathbf{u}$ and $\mathbf{v}$ in the Nitsche terms
are replaced by their extensions $\mathbf{u}_p$ and $\mathbf{v}_p$ evaluated at the projected points on the true
boundary. This yields:
\begin{equation}
    \begin{split}
        \int_{\Omega \pr} \nabla \mathbf{v} : \boldsymbol{\sigma}(\mathbf{u},p) \, dx - \int_{\Omega \pr} q \, \nabla \cdot \mathbf{u} \, dx
        - \int_{\Gamma \pr} \mathbf{n}\pr \cdot \boldsymbol{\sigma}(\mathbf{u},p)\cdot\mathbf{v} \, ds                                                                   \\
        -\int_{\Gamma \pr} \mathbf{n}\pr \cdot \boldsymbol{\sigma}(\mathbf{v},q)\cdot\mathbf{u}_p \, ds + \alpha \int_{\Gamma \pr} \mathbf{u}_p \cdot \mathbf{v}_p \, ds \\
        = \int_{\Omega \pr} \mathbf{v} \cdot \mathbf{f} \, dx
        - \int_{\Gamma \pr} \mathbf{n}\pr \cdot \boldsymbol{\sigma}(\mathbf{v},q)\cdot\mathbf{u}_d \, ds + \alpha \int_{\Gamma \pr} \mathbf{u}_d \cdot \mathbf{v}_p \, ds .
    \end{split}
    \label{eq:sbm_stokes_compact}
\end{equation}
Note that $\mathbf{u}_p$ and $\mathbf{v}_p$ need not be extended in the same manner; the specific choices made in our
implementation are discussed in Section~\ref{sec:sbm_assembly}, and the asymptotic-symmetry analysis of
Section~\ref{sec:theory} (Lemma~\ref{lem:perturbation}) is carried out for that specific extension pair. Expanding $\boldsymbol{\sigma}$ into its definition
recovers the full discretisation of \cite{atallah2020analysis}:
\begin{equation}
    \begin{split}
        \int_{\Omega \pr} \nabla \mathbf{v} : 2 \varepsilon(\mathbf{u}) \, dx - \int_{\Omega \pr} (\nabla \cdot \mathbf{v}) \, p \, dx - \int_{\Omega \pr} q \, \nabla \cdot \mathbf{u} \, dx
        - \int_{\Gamma \pr} \mathbf{n}\pr \cdot (2 \varepsilon(\mathbf{u})\mathbf{v}) \, ds + \int_{\Gamma \pr} (\mathbf{n}\pr \cdot \mathbf{v}) \, p \, ds                                                                     \\
        -\int_{\Gamma \pr} \mathbf{n}\pr \cdot (2 \varepsilon(\mathbf{v})\mathbf{u}_p) \, ds + \int_{\Gamma \pr} (\mathbf{n}\pr \cdot \mathbf{u}_p) \, q \, ds + \alpha \int_{\Gamma \pr} \mathbf{u}_p \cdot \mathbf{v}_p \, ds \\
        = \int_{\Omega \pr} \mathbf{v} \cdot \mathbf{f} \, dx
        - \int_{\Gamma \pr} \mathbf{n}\pr \cdot (2 \varepsilon(\mathbf{v})\mathbf{u}_d) \, ds + \int_{\Gamma \pr} (\mathbf{n}\pr \cdot \mathbf{u}_d) \, q \, ds + \alpha \int_{\Gamma \pr} \mathbf{u}_d \cdot \mathbf{v}_p \, ds .
    \end{split}
    \label{eq:sbm_stokes_full}
\end{equation}
For the continuous-pressure formulation with the pressure-gradient stabilisation of
\cite{atallah2020analysis}, this formulation is coercive, continuous, and satisfies the LBB inf-sup
condition. For the inf-sup stable pairs employed in this work (the discontinuous-pressure
$\pairDGthree$ and $\pairDGtwo$ pairs and the continuous-pressure Taylor--Hood pairs), we assume that
the corresponding well-posedness result of \cite{atallah2020analysis} continues to hold; a proof for
these specific pairs is beyond the scope of the present work. We make this precise as
Assumption~\ref{ass:wellposed} below and verify the resulting solver behaviour experimentally in
Section~\ref{sec:results}.

\subsubsection*{Algebraic stability and the inf-sup condition}

The numerical solution of the incompressible Stokes equations poses challenges absent from symmetric positive-definite
problems such as the Poisson equation. The coupling of velocity and pressure through the incompressibility constraint
leads to a saddle-point system of the form
\begin{equation}
    \begin{bmatrix} A & B^T \\ B & 0 \end{bmatrix}
    \begin{bmatrix} \mathbf{u} \\ \mathbf{p} \end{bmatrix}
    =
    \begin{bmatrix} \mathbf{f} \\ \mathbf{g} \end{bmatrix},
\end{equation}
where $A$ is the discrete viscous operator, $B$ is the discrete divergence operator, and $\mathbf{u}$, $\mathbf{p}$
denote the coefficient vectors of the discrete velocity and pressure. The zero diagonal block renders the system
indefinite, ruling out standard iterative methods designed for elliptic problems, such as the Conjugate Gradient
method.

A fundamental requirement for the well-posedness of such systems is the \emph{inf-sup condition}, also known as the
Ladyzhenskaya--Babu\v{s}ka--Brezzi (LBB) condition \cite{Elman2014-gj}. For a mixed discretisation with velocity space
$V_h$ and pressure space $Q_h$, the discrete inf-sup condition requires the existence of a constant $\beta > 0$,
independent of the mesh size $h$, such that
\begin{equation}
    \inf_{q_h \in Q_h} \sup_{\mathbf{v}_h \in V_h}
    \frac{b(\mathbf{v}_h, q_h)}{\|\mathbf{v}_h\|_{V} \|q_h\|_{Q}} \geq \beta > 0,
\end{equation}
where $b(\mathbf{v}_h, q_h) = -\int_\Omega q_h \nabla \cdot \mathbf{v}_h \, dx$ is the bilinear form associated with
the incompressibility constraint.

The inf-sup condition ensures that the pressure space $Q_h$ is not ``too large'' relative to the velocity space $V_h$:
every pressure mode must be ``seen'' by some velocity field through the divergence operator. When the condition is
violated, $B^T$ has a non-trivial kernel within $Q_h$, leading to spurious pressure modes and a rank-deficient
saddle-point system. The condition restricts the admissible velocity--pressure pairings; the classical Taylor--Hood
element ($\mathbb{Q}_k / \mathbb{Q}_{k-1}$, $k \geq 2$) is a well-known inf-sup stable pair on quadrilateral meshes
\cite{Elman2014-gj}.

In the SBM setting, verifying the inf-sup condition requires additional care, because the boundary terms introduced by
the Taylor expansion at the surrogate boundary modify both $A$ and $B$ relative to their body-fitted counterparts.
Atallah, Canuto, and Scovazzi \cite{atallah2020analysis} proved that the SBM-Stokes formulation satisfies a discrete
inf-sup condition for appropriate element pairs, including the $\pairDGthree$ and $\pairDGtwo$ pairs used in this work.
High-order pairs are particularly attractive for SBM, since the accuracy of the boundary shift relies on the order of
the Taylor expansion; however, they also increase the number of degrees of freedom and the density of the system matrix
\cite{atallah2022high}, motivating the robust preconditioning strategies discussed next.

\subsubsection*{Preconditioning and scalability}
The high-order discretisation of the SBM-Stokes system leads to large, sparse linear systems. As the mesh is refined
to capture geometric detail, the memory and time requirements of direct solvers such as sparse LU factorisation scale
poorly, making them impractical for large-scale applications. Krylov subspace methods such as GMRES or MINRES are then
the methods of choice; their convergence, however, depends on the conditioning of the system matrix, which
deteriorates as the mesh size $h$ decreases.

Preconditioning is essential. According to the operator-preconditioning framework of Mardal and Winther
\cite{PreconditioningPDEs}, an effective preconditioner should reflect the Hilbert-space structure of the underlying
partial differential equation, yielding preconditioned operators with $h$-independent spectral bounds.

For the Stokes problem, this is commonly realised through block preconditioning \cite{BlockPreconditioners,
    BlockPrec2}. The principal difficulty is approximating the Schur complement, which encodes the pressure--velocity
coupling. A well-established remedy is to use the pressure mass matrix as a Schur complement approximation, which
yields mesh-independent convergence for standard Stokes discretisations \cite{Elman2014-gj}. While geometric multigrid
techniques for SBM have recently been developed \cite{wichrowski2025geometric, shypatch}, the application of
block-diagonal and block-triangular preconditioners to high-order SBM-Stokes discretisations has remained largely
unexplored. This work demonstrates that iteration counts remain low and stable under mesh refinement, preserving the
efficiency of the SBM approach as the geometric complexity increases.

\subsection{Block preconditioners}

The SBM discretisation of the Stokes equations yields large, ill-conditioned saddle-point systems, particularly under
mesh refinement. The well-posedness of the formulation and the LBB inf-sup condition provide the algebraic structure
that allows block preconditioners to achieve mesh-independent convergence, as established by the theory of Mardal and
Winther \cite{PreconditioningPDEs}.

For a saddle-point system
\begin{equation}
    \begin{bmatrix}
        A & B^T \\
        B & 0
    \end{bmatrix}
    \begin{bmatrix}
        \mathbf{u} \\
        \mathbf{p}
    \end{bmatrix}
    =
    \begin{bmatrix}
        \mathbf{f} \\
        \mathbf{g}
    \end{bmatrix},
\end{equation}
with $A$ invertible on the kernel of $B$, a family of block preconditioners can be defined \cite{BlockPrec2} as
\begin{equation}
    \mathcal{P} =
    \begin{bmatrix}
        I          &   \\
        cBA_0^{-1} & I
    \end{bmatrix}
    \begin{bmatrix}
        A_{0} &        \\
              & -S_{0}
    \end{bmatrix}
    \begin{bmatrix}
        I & dA_0^{-1}B^T \\
          & I
    \end{bmatrix},
\end{equation}
where $A_0$ and $S_0$ are symmetric positive-definite matrices and $c,d$ are prescribed real numbers. For
mesh-independent convergence, $A_0$ and $S_0$ should be spectrally equivalent to $A$ and to the Schur complement
$S = BA^{-1}B^T$, respectively \cite{PreconditioningPDEs}; under these conditions, such preconditioners guarantee
$h$-independent convergence rates for inf-sup stable discretisations \cite{BlockPreconditioners}. Setting $c=0$ and
$d=1$ yields the block upper-triangular preconditioner
\begin{equation}
    \mathcal{P} =
    \begin{bmatrix}
        A_{0} & B^T    \\
              & -S_{0}
    \end{bmatrix} .
\end{equation}
Efficient approximate inverses are available for both blocks. In this work we use the exact velocity block,
$A_0 = A$; the framework is nevertheless compatible with the geometric multigrid (GMG) strategies and \emph{Shy Patch}
smoothers introduced in \cite{wichrowski2025geometric,shypatch}, which could replace the exact solve. For $S_0$ we use
the pressure mass matrix $M$, which is inexpensive to assemble and to invert. The preconditioner employed in this work
is
\begin{equation}
    \mathcal{P} =
    \begin{bmatrix}
        A & B^T \\
          & -M
    \end{bmatrix} .
    \label{eq:prec_used}
\end{equation}

\subsection{Field-of-values analysis of the preconditioned system}
\label{sec:theory}

The classical operator-preconditioning theory \cite{PreconditioningPDEs} and the spectral analysis of block
preconditioners \cite{BlockPreconditioners} are formulated for symmetric saddle-point problems. The SBM-Stokes system
is non-symmetric: the Taylor corrections act on trial and test functions differently, so the discrete divergence blocks
are not transposes of each other and the velocity block has a non-trivial skew-symmetric part. The full coupled
bilinear form is \emph{not} coercive (numerical evaluation of the eigenvalues of its symmetric part shows that they may
be negative \cite{atallah2020analysis}), so a naive energy-norm argument on the whole system is not available. The
appropriate tool for non-symmetric problems is the field-of-values (FOV) analysis of the preconditioned operator
\cite{starke1997fov,klawonn1999fov,loghin2004analysis}, which we develop here. The key structural observation is that
\emph{every} non-symmetric term in the SBM formulation carries at least one factor of the distance vector $\mathbf{d}$
between the surrogate and true boundaries, and is asymptotically small under mesh refinement.

We write the discrete problem abstractly as: find $[\mathbf{u}_h, p_h] \in V_h \times Q_h$ such that
\begin{equation}
    \mathcal{B}([\mathbf{u}_h,p_h];[\mathbf{v}_h,q_h])
    = a(\mathbf{u}_h,\mathbf{v}_h) + b_1(\mathbf{v}_h,p_h) + b_2(\mathbf{u}_h,q_h)
    = \ell(\mathbf{v}_h,q_h)
    \qquad \forall [\mathbf{v}_h,q_h] \in V_h \times Q_h,
\end{equation}
where $a$ collects the viscous volume term and all velocity boundary terms,
and $b_1$, $b_2$ collect the pressure--velocity couplings, including their SBM boundary
corrections. We equip $V_h$ with the energy norm $\|\cdot\|_a$ of
\cite{atallah2020analysis}, which is equivalent to the $H^1$-norm on $V_h$, and $Q_h$ with
the scaled $L^2$-norm $\|q\|_Q = \|q/\sqrt{2\mu}\|_{0,\tilde\Omega}$, whose Gram matrix is
the pressure mass matrix $M$.

Rather than tying the analysis to a particular element pair, we make the following general assumption.

\begin{assumption}[well-posed pair]
    \label{ass:wellposed}
    There exist constants $C_a, C_A, \beta, C_b > 0$, independent of $h$, such that
    \begin{align}
        a(\mathbf{v}_h,\mathbf{v}_h)                                  & \geq C_a \|\mathbf{v}_h\|_a^2 ,
                                                                      & |a(\mathbf{u}_h,\mathbf{v}_h)|
                                                                      & \leq C_A \|\mathbf{u}_h\|_a \|\mathbf{v}_h\|_a ,
        \label{eq:ass_a}                                                                                                     \\
        \adjustlimits \inf_{q_h \in Q_h} \sup_{\mathbf{v}_h \in V_h}
        \frac{b_2(\mathbf{v}_h,q_h)}{\|\mathbf{v}_h\|_a \, \|q_h\|_Q} & \geq \beta ,
                                                                      & |b_i(\mathbf{v}_h,q_h)|
                                                                      & \leq C_b \|\mathbf{v}_h\|_a \|q_h\|_Q , \quad i=1,2.
        \label{eq:ass_b}
    \end{align}
\end{assumption}

For the SBM-Stokes formulation with continuous pressure (and the pressure-gradient stabilisation of
\cite{atallah2020analysis}), Assumption~\ref{ass:wellposed} is verified in \cite{atallah2020analysis}: coercivity and
continuity of $a$ are established in their Theorem on coercivity (for Nitsche, tangential, and divergence penalty
parameters $\alpha,\beta,\gamma_2$ large enough), and the uniform inf-sup follows from their LBB theorem, whose proof
constructs the supremising velocity via a Scott--Zhang quasi-interpolant vanishing on the surrogate Dirichlet boundary.
For the inf-sup stable pairs employed in this work (the discontinuous-pressure $\pairDGthree$ and $\pairDGtwo$ pairs
and the continuous-pressure Taylor--Hood pairs), a proof of Assumption~\ref{ass:wellposed} is beyond the present scope;
we take it as a hypothesis, on the expectation that the argument of \cite{atallah2020analysis} extends to these pairs,
and verify the resulting solver behaviour experimentally in Section~\ref{sec:results}.

The second ingredient is the geometric resolution assumption of \cite{atallah2020analysis} (their Assumption 4): there
exist $c_d > 0$ and $\zeta > 0$ such that
\begin{equation}
    \|\mathbf{d}(\tilde{\mathbf{x}})\| \leq c_d \, h_T \hat{h}_T^{\zeta},
    \qquad \hat{h}_T = h_T / l(\tilde\Omega),
    \label{eq:ass_d}
\end{equation}
for all points $\tilde{\mathbf{x}}$ on the surrogate boundary. This expresses that the
distance between the surrogate and true boundaries decays marginally faster than $h$, and
is the natural quantification of the mesh being ``fine enough to resolve the geometry''.

\begin{lemma}[asymptotic symmetry]
    \label{lem:perturbation}
    Let \eqref{eq:ass_d} hold, let the trial and test functions be extended by the
    elementwise polynomial extrapolation $\mathbf{u}^F$ and the first-order Taylor
    extension $\mathbf{v}^T$ of Section~\ref{sec:sbm_assembly}, and let the Nitsche
    penalty scale as $\alpha \leq c_\alpha\, 2\mu\, h_T^{-1}$. In the symmetric-gradient
    case $\boldsymbol{\nabla}_s = \varepsilon$ assume in addition the smallness condition
    $\hat{h}_{\tilde\Gamma_D} \leq \hat{h}_0$; in the regular-gradient case
    $\boldsymbol{\nabla}_s = \tfrac12\nabla$ no such restriction is required (see Step~1 of
    the proof). Set $\varepsilon_h = C \hat{h}_{\tilde\Gamma_D}^{\zeta}$
    for a constant $C$ independent of $h$. Then
    \begin{align}
        |a(\mathbf{u}_h,\mathbf{v}_h) - a(\mathbf{v}_h,\mathbf{u}_h)|
         & \leq \varepsilon_h \, \|\mathbf{u}_h\|_a \|\mathbf{v}_h\|_a ,
        \label{eq:skew_a}                                                \\
        |b_1(\mathbf{v}_h,q_h) - b_2(\mathbf{v}_h,q_h)|
         & \leq \varepsilon_h \, \|\mathbf{v}_h\|_a \|q_h\|_Q .
        \label{eq:skew_b}
    \end{align}
\end{lemma}

\begin{proof}
    Throughout, $T$ denotes a generic element with a face on the surrogate boundary, $C$
    a generic positive constant depending only on the shape regularity of the mesh, the
    polynomial degrees, and the constants $c_d$, $\zeta$, $c_\alpha$ introduced below
    (but never on $h$), and $\hat h := \hat h_{\tilde\Gamma_D} = \max\{\hat h_T :
        \partial T \cap \Gamma\pr \neq \emptyset\}$. We abbreviate
    $\langle \cdot,\cdot \rangle := \langle \cdot,\cdot \rangle_{\Gamma\pr}$.

    Recall from Section~\ref{sec:sbm_assembly} the two extensions in use: trial functions are extended by elementwise
    polynomial extrapolation, $\mathbf{u}^F(\bar{\mathbf{x}}) := \mathbf{u}\big(\bar{\mathbf{x}} +
        \mathbf{d}(\bar{\mathbf{x}})\big)$, and test functions by the first-order Taylor expansion $\mathbf{v}^T := \mathbf{v}
        + (\nabla \mathbf{v})\,\mathbf{d}$. We write $\delta\mathbf{u} := \mathbf{u}^F - \mathbf{u}$ for the extension
    increment and $\mathbf{r}_u := \mathbf{u}^F - \mathbf{u}^T$ for the first-order Taylor remainder. Since $\mathbf{u}$ is
    an elementwise polynomial, on the segment $S(\bar{\mathbf{x}}) := [\bar{\mathbf{x}}, \bar{\mathbf{x}} + \mathbf{d}]$ we
    have the pointwise bounds
    \begin{equation}
        |\delta\mathbf{u}(\bar{\mathbf{x}})| \leq \|\mathbf{d}\| \max_{S(\bar{\mathbf{x}})} |\nabla \mathbf{u}| ,
        \qquad
        |\mathbf{r}_u(\bar{\mathbf{x}})| \leq \tfrac12 \|\mathbf{d}\|^2 \max_{S(\bar{\mathbf{x}})} |D^2 \mathbf{u}| .
        \label{eq:pointwise_taylor}
    \end{equation}
    We use the following standard facts:
    \begin{itemize}

        \item[(i)] the elementwise discrete trace inequality
              $\|\varphi\|_{0;\partial T \cap \Gamma\pr}^2 \leq C h_T^{-1} \|\varphi\|_{0;T}^2$ for
              elementwise polynomials $\varphi$;
        \item[(ii)] by \eqref{eq:ass_d} and shape regularity, the
              segment $S(\bar{\mathbf{x}})$ is contained in a uniformly bounded enlargement of $T$,
              on which the equivalence of norms on finite-dimensional polynomial spaces and inverse
              estimates give $\max_{S} |\varphi| \leq C h_T^{-d/2} \|\varphi\|_{0;T}$ and
              $\max_{S} |\nabla\varphi| \leq C h_T^{-1-d/2} \|\varphi\|_{0;T}$;

        \item[(iii)] the
              resolution assumption $\|\mathbf{d}\|_{L^\infty(\partial T \cap \Gamma\pr)} \leq c_d
                  h_T \hat h^{\zeta}$;
        \item[(iv)] the properties of the energy norm of
              \cite{atallah2020analysis} (their Eq.~(21)):
              $\sqrt{2\mu}\,\|\varepsilon(\mathbf{v})\|_{0;\tilde\Omega} \leq \|\mathbf{v}\|_a$ and
              the Taylor-extended trace term $\sqrt{2\mu/h}\,\|\mathbf{v}^T\|_{0;\Gamma\pr} \leq
                  \|\mathbf{v}\|_a$ are constituents of the norm, so that $\alpha^{1/2}
                  \|\mathbf{v}^T\|_{0;\Gamma\pr} \leq \sqrt{c_\alpha}\,\|\mathbf{v}\|_a$ for a penalty
              parameter scaling as $\alpha \leq c_\alpha\, 2\mu\, h_T^{-1}$.
    \end{itemize}
    We will also use the
    uniform gradient bound
    \begin{equation}
        \sqrt{2\mu}\,\|\nabla \mathbf{v}\|_{0;\tilde\Omega} \leq C_K \|\mathbf{v}\|_a ,
        \label{eq:korn_uniform}
    \end{equation}
    with $C_K$ independent of $h$, established in Step 1 below. We stress that
    \eqref{eq:korn_uniform} does not follow from the equivalence of $\|\cdot\|_a$ with
    the full $H^1$-norm proved in \cite{atallah2020analysis} (their Proposition~2),
    since that equivalence is not uniform with respect to the mesh size; it is only the
    gradient seminorm that admits a uniform bound. This is the sole point in the analysis
    where the symmetric structure of $\boldsymbol{\nabla}_s = \varepsilon$ is used. For the
    regular-gradient choice $\boldsymbol{\nabla}_s = \tfrac12\nabla$ of
    \eqref{eq:stress_tensor}, the energy norm controls $\|\nabla\mathbf{v}\|$ directly, so
    \eqref{eq:korn_uniform} holds trivially with $C_K = 1$, Step 1 below is vacuous, and the
    associated smallness restriction $\hat h \leq \hat h_0$ is not needed; all remaining
    estimates are unchanged.

    \emph{Step 1: weighted boundary estimates.} For any elementwise polynomial $\varphi$,
    (i) and (iii) give the weighted trace estimate
    \begin{equation}
        \big\| \|\mathbf{d}\|^{1/2} \varphi \big\|_{0;\Gamma\pr}^2
        \;\leq\; \sum_T c_d h_T \hat h^{\zeta} \cdot C h_T^{-1} \|\varphi\|_{0;T}^2
        \;\leq\; C \hat h^{\zeta} \|\varphi\|_{0;\tilde\Omega}^2 .
        \label{eq:weighted_trace}
    \end{equation}
    For the extension increment, \eqref{eq:pointwise_taylor}, (ii), (iii) and
    $|\partial T \cap \Gamma\pr| \leq C h_T^{d-1}$ yield
    \begin{equation*}
        \big\| \|\mathbf{d}\|^{-1/2} \delta\mathbf{u} \big\|_{0;\partial T \cap \Gamma\pr}^2
        \leq \|\mathbf{d}\|_{L^\infty} \max_S |\nabla\mathbf{u}|^2 \, |\partial T \cap \Gamma\pr|
        \leq c_d h_T \hat h^\zeta \cdot C h_T^{-d} \|\nabla \mathbf{u}\|_{0;T}^2 \cdot C h_T^{d-1} ,
    \end{equation*}
    and summation over the boundary elements gives
    \begin{equation}
        \big\| \|\mathbf{d}\|^{-1/2} \delta\mathbf{u} \big\|_{0;\Gamma\pr}^2
        \leq C \hat h^{\zeta} \|\nabla \mathbf{u}\|_{0;\tilde\Omega}^2 ;
        \label{eq:increment_est}
    \end{equation}
    the same bound holds for the Taylor increment $(\nabla\mathbf{u})\mathbf{d}$,
    directly from \eqref{eq:weighted_trace}. Analogously, for the second-order remainder,
    \eqref{eq:pointwise_taylor} together with
    $$\max_S|D^2\mathbf{u}| \leq C h_T^{-1-d/2}\|\nabla\mathbf{u}\|_{0;T}$$ gives
    $\|\mathbf{r}_u\|_{0;\partial T \cap \Gamma\pr}^2 \leq C (c_d h_T \hat h^{\zeta})^4
        h_T^{-2-d} \|\nabla \mathbf{u}\|_{0;T}^2 \, h_T^{d-1} = C \hat h^{4\zeta} h_T
        \|\nabla \mathbf{u}\|_{0;T}^2$, whence, by the scaling of $\alpha$,
    \begin{equation}
        \alpha \, \|\mathbf{r}_u\|_{0;\Gamma\pr}^2
        \leq C \hat h^{4\zeta} \, 2\mu \, \|\nabla \mathbf{u}\|_{0;\tilde\Omega}^2 .
        \label{eq:remainder_est}
    \end{equation}
    We can now prove \eqref{eq:korn_uniform}. Korn's inequality with boundary control
    \cite{atallah2020analysis}
    (their Eq.~(A.11a))
    gives
    $\|\nabla\mathbf{v}\|_{0;\tilde\Omega}^2 \leq C \big(
        \|\varepsilon(\mathbf{v})\|_{0;\tilde\Omega}^2 + l(\tilde\Omega)^{-1}
        \|\mathbf{v}\|_{0;\Gamma\pr}^2 \big)$. Splitting the trace as $\mathbf{v} =
        \mathbf{v}^T - (\nabla\mathbf{v})\mathbf{d}$, the first part is controlled by the
    penalty constituent of the norm, $l(\tilde\Omega)^{-1} \|\mathbf{v}^T\|_{0;\Gamma\pr}^2
        \leq \hat h \, h^{-1} \|\mathbf{v}^T\|_{0;\Gamma\pr}^2 \leq \hat h \, (2\mu)^{-1}
        \|\mathbf{v}\|_a^2$, while \eqref{eq:weighted_trace} and (iii) bound the second part
    by $l(\tilde\Omega)^{-1} \|(\nabla\mathbf{v})\mathbf{d}\|_{0;\Gamma\pr}^2 \leq C \hat
        h^{1+2\zeta} \|\nabla\mathbf{v}\|_{0;\tilde\Omega}^2$. Hence
    \begin{equation*}
        2\mu\, \|\nabla\mathbf{v}\|_{0;\tilde\Omega}^2
        \leq C \|\mathbf{v}\|_a^2 + C \hat h\, 2\mu\, \|\nabla\mathbf{v}\|_{0;\tilde\Omega}^2 ,
    \end{equation*}
    and the last term can be absorbed into the left-hand side once $\hat h \leq \hat
        h_0$, with $\hat h_0$ depending only on the constants in (i)--(iv); this is no
    restriction, since the conclusions of the lemma are only invoked in the asymptotic
    regime of Theorem~\ref{thm:fov}.

    \emph{Step 2: the divergence mismatch \eqref{eq:skew_b}.} From
    \eqref{eq:sbm_stokes_full},
    \begin{equation*}
        b_1(\mathbf{v},q) = -\int_{\Omega\pr} (\nabla\cdot\mathbf{v})\, q \, dx + \langle \mathbf{n}\pr \cdot \mathbf{v},\, q \rangle ,
        \qquad
        b_2(\mathbf{v},q) = -\int_{\Omega\pr} (\nabla\cdot\mathbf{v})\, q \, dx + \langle \mathbf{n}\pr \cdot \mathbf{v}^F,\, q \rangle ,
    \end{equation*}
    so that $b_1(\mathbf{v},q) - b_2(\mathbf{v},q) = -\langle \mathbf{n}\pr \cdot
        \delta\mathbf{v},\, q\rangle$. The Cauchy--Schwarz inequality with the weights
    $(2\mu/\|\mathbf{d}\|)^{1/2}$ and $(\|\mathbf{d}\|/2\mu)^{1/2}$, followed by
    \eqref{eq:increment_est}, \eqref{eq:weighted_trace} and \eqref{eq:korn_uniform}, gives
    \begin{equation*}
        |\langle \mathbf{n}\pr \cdot \delta\mathbf{v}, q \rangle|
        \leq \big\|(2\mu/\|\mathbf{d}\|)^{1/2}\, \delta \mathbf{v}\big\|_{0;\Gamma\pr}
        \big\|(\|\mathbf{d}\|/2\mu)^{1/2}\, q\big\|_{0;\Gamma\pr}
        \leq C \hat h^{\zeta} \, \sqrt{2\mu}\,\|\nabla\mathbf{v}\|_{0;\tilde\Omega} \, \|q\|_Q
        \leq C \hat h^{\zeta} \|\mathbf{v}\|_a \|q\|_Q ,
    \end{equation*}
    which is \eqref{eq:skew_b}.

    \emph{Step 3: the skew part of $a$, \eqref{eq:skew_a}.} From
    \eqref{eq:sbm_stokes_full},
    \begin{equation*}
        a(\mathbf{u},\mathbf{v}) = \int_{\Omega\pr} 2\mu\, \varepsilon(\mathbf{u}) : \varepsilon(\mathbf{v}) \, dx
        - \langle 2\mu\, \mathbf{n}\pr \cdot \varepsilon(\mathbf{u}),\, \mathbf{v} \rangle
        - \langle 2\mu\, \mathbf{n}\pr \cdot \varepsilon(\mathbf{v}),\, \mathbf{u}^F \rangle
        + \alpha \langle \mathbf{u}^F,\, \mathbf{v}^T \rangle ,
    \end{equation*}
    where the volume term is symmetric since $\nabla\mathbf{v} :
        \varepsilon(\mathbf{u}) = \varepsilon(\mathbf{v}) : \varepsilon(\mathbf{u})$. Hence
    \begin{equation*}
        a(\mathbf{u},\mathbf{v}) - a(\mathbf{v},\mathbf{u})
        = \langle 2\mu\, \mathbf{n}\pr\cdot\varepsilon(\mathbf{u}),\, \delta\mathbf{v} \rangle
        - \langle 2\mu\, \mathbf{n}\pr\cdot\varepsilon(\mathbf{v}),\, \delta\mathbf{u} \rangle
        + \alpha \big( \langle \mathbf{u}^F, \mathbf{v}^T \rangle - \langle \mathbf{v}^F, \mathbf{u}^T \rangle \big) .
    \end{equation*}
    The consistency mismatches are estimated as in Step 2, with the roles of the weights
    exchanged:
    \begin{equation*}
        |\langle 2\mu\, \mathbf{n}\pr\cdot\varepsilon(\mathbf{u}),\, \delta\mathbf{v} \rangle|
        \leq \big\| (2\mu \|\mathbf{d}\|)^{1/2} \varepsilon(\mathbf{u})\,\mathbf{n}\pr \big\|_{0;\Gamma\pr}
        \big\| (2\mu/\|\mathbf{d}\|)^{1/2} \delta\mathbf{v} \big\|_{0;\Gamma\pr}
        \leq C \hat h^{\zeta} \|\mathbf{u}\|_a \|\mathbf{v}\|_a ,
    \end{equation*}
    by \eqref{eq:weighted_trace} applied to $\varepsilon(\mathbf{u})$,
    \eqref{eq:increment_est}, and (iv). For the penalty mismatch, substituting
    $\mathbf{u}^F = \mathbf{u}^T + \mathbf{r}_u$ and
    $\mathbf{v}^F = \mathbf{v}^T + \mathbf{r}_v$ and using the symmetry of
    $\langle \mathbf{u}^T, \mathbf{v}^T \rangle$,
    \begin{equation*}
        \alpha \big( \langle \mathbf{u}^F, \mathbf{v}^T \rangle - \langle \mathbf{v}^F, \mathbf{u}^T \rangle \big)
        = \alpha \langle \mathbf{r}_u, \mathbf{v}^T \rangle - \alpha \langle \mathbf{r}_v, \mathbf{u}^T \rangle ,
    \end{equation*}
    and by the Cauchy--Schwarz inequality, \eqref{eq:remainder_est}, (iv) and \eqref{eq:korn_uniform},
    \begin{equation*}
        \alpha\, |\langle \mathbf{r}_u, \mathbf{v}^T \rangle|
        \leq \big( \alpha^{1/2} \|\mathbf{r}_u\|_{0;\Gamma\pr} \big)
        \big( \alpha^{1/2} \|\mathbf{v}^T\|_{0;\Gamma\pr} \big)
        \leq C \hat h^{2\zeta} \, \sqrt{2\mu}\, \|\nabla\mathbf{u}\|_{0;\tilde\Omega} \, \|\mathbf{v}\|_a
        \leq C \hat h^{2\zeta} \|\mathbf{u}\|_a \|\mathbf{v}\|_a .
    \end{equation*}
    Since $\hat h \leq 1$ we have $\hat h^{2\zeta} \leq \hat h^{\zeta}$, and collecting
    the estimates proves \eqref{eq:skew_a} with $\varepsilon_h = C \hat h^{\zeta}$. We
    remark that if trial and test functions are extended in the same manner, the penalty
    mismatch vanishes identically and the same proof applies with $\delta$ replaced by
    the corresponding increment; on a body-fitted mesh ($\mathbf{d} \equiv \mathbf{0}$)
    all mismatch terms vanish and the formulation is symmetric.
\end{proof}

Lemma~\ref{lem:perturbation} states that the discrete system matrix admits the splitting
\begin{equation}
    \mathcal{A} =
    \begin{bmatrix}
        A & B_1^T \\ B_2 & 0
    \end{bmatrix}
    =
    \underbrace{\begin{bmatrix}
            A_s & B^T \\ B & 0
        \end{bmatrix}}_{\textstyle \mathcal{A}_s}
    + \; \mathcal{E}_h ,
    \label{eq:splitting}
\end{equation}
where $A_s$ is the symmetric part of $A$, $B = \tfrac12(B_1 + B_2)$, and the perturbation
$\mathcal{E}_h$ has norm $O(\varepsilon_h)$ with respect to the norms of
Assumption~\ref{ass:wellposed}. The symmetric leading part $\mathcal{A}_s$ satisfies the
classical Brezzi conditions with constants $C_a$, $C_A$,
$\beta - \varepsilon_h/2$, $C_b$, so for $h$ small enough it is a standard, uniformly
stable Stokes-type saddle-point matrix. (The velocity coercivity constant is
unchanged, since $\mathbf{v}^{T} A_s \mathbf{v} = \mathbf{v}^{T} A \mathbf{v} =
    a(\mathbf{v},\mathbf{v}) \geq C_a \|\mathbf{v}\|_a^2$; only the inf-sup constant of the
symmetrised coupling $b = \tfrac12(b_1+b_2)$ is reduced, by at most
$\varepsilon_h/2$, as a consequence of \eqref{eq:skew_b}.)

\begin{theorem}[mesh-independent FOV bounds]
    \label{thm:fov}
    Let Assumption~\ref{ass:wellposed}, \eqref{eq:ass_d}, and the extension and penalty
    assumptions of Lemma~\ref{lem:perturbation} hold, and let
    \begin{equation*}
        \mathcal{P} =
        \begin{bmatrix}
            A & B_1^T \\ & -M
        \end{bmatrix}
    \end{equation*}
    be the block upper-triangular preconditioner with exact velocity block. Then there
    exist $h_0 > 0$ and constants $0 < \gamma \leq \Gamma$, depending only on
    $C_a, C_A, \beta, C_b$, such that for all $h \leq h_0$ the matrix $\mathcal{P}$ is
    invertible and the field of values of the right-preconditioned operator
    $\mathcal{A}\mathcal{P}^{-1}$ satisfies, in the inner product induced by the
    block-diagonal matrix $\mathcal{H} = \operatorname{diag}(A_s^{-1}, M^{-1})$,
    \begin{equation}
        \gamma \leq
        \frac{\langle \mathcal{A}\mathcal{P}^{-1} y, y \rangle_{\mathcal{H}}}{\langle y, y \rangle_{\mathcal{H}}}
        \qquad \text{and} \qquad
        \|\mathcal{A}\mathcal{P}^{-1} y\|_{\mathcal{H}} \leq \Gamma \|y\|_{\mathcal{H}}
        \qquad \forall y \neq 0 .
        \label{eq:fov}
    \end{equation}
    GMRES applied to the right-preconditioned system in the
    $\mathcal{H}$-inner product reduces the residual by a factor of at least
    $\left(1 - \gamma^2/\Gamma^2\right)^{1/2}$ per iteration, independently of the mesh
    size \cite{starke1997fov,eisenstat1983variational}.
\end{theorem}

\begin{proof}
    Set $D := \operatorname{diag}(A_s^{1/2}, M^{1/2})$, so that $\mathcal{H} = D^{-2}$.
    For any matrix $T$ and $y = Dz$ we have
    $\langle Ty, y \rangle_{\mathcal{H}} = \langle D^{-1}TD\, z, z\rangle$ and
    $\|y\|_{\mathcal{H}} = \|z\|$: the $\mathcal{H}$-field of values of $T$ is the
    Euclidean field of values of $D^{-1}TD$, and we work with the transformed matrices
    throughout. We record the norm equivalences implied by \eqref{eq:ass_a},
    \begin{equation}
        C_a \|\mathbf{v}\|_a^2 \leq \|\mathbf{v}\|_{A_s}^2 = a(\mathbf{v},\mathbf{v}) \leq C_A \|\mathbf{v}\|_a^2 ,
        \qquad \|q\|_M = \|q\|_Q ,
        \label{eq:norm_equiv}
    \end{equation}
    and fix $h_1$ such that $\varepsilon_h \leq \beta$ for $h \leq h_1$, so that the
    inf-sup constant below does not degenerate.

    \emph{Step 1: the symmetrised pressure--velocity coupling.} Let
    $B = \tfrac12(B_1 + B_2)$ as in \eqref{eq:splitting}, with associated form
    $b = \tfrac12(b_1 + b_2)$, and set $K := M^{-1/2} B A_s^{-1/2}$. By
    Lemma~\ref{lem:perturbation}, $|b(\mathbf{v},q) - b_2(\mathbf{v},q)| \leq \tfrac12
        \varepsilon_h \|\mathbf{v}\|_a \|q\|_Q$, so $b$ inherits from \eqref{eq:ass_b} the
    continuity constant $C_b$ and, for $h \leq h_1$, the inf-sup constant
    $\beta_h := \beta - \varepsilon_h/2 \geq \beta/2$. For any $q$ with
    $\hat q = M^{1/2} q$,
    \begin{equation*}
        \|K^T \hat q\| = \sup_{\mathbf{v}} \frac{q^T B \mathbf{v}}{\|\mathbf{v}\|_{A_s}}
        = \sup_{\mathbf{v}} \frac{b(\mathbf{v}, q)}{\|\mathbf{v}\|_{A_s}} ,
    \end{equation*}
    so \eqref{eq:norm_equiv} gives
    $\sigma_- \|\hat q\| \leq \|K^T \hat q\| \leq \sigma_+ \|\hat q\|$ with
    \begin{equation*}
        \sigma_- := \frac{\beta}{2\sqrt{C_A}} , \qquad \sigma_+ := \frac{C_b}{\sqrt{C_a}} .
    \end{equation*}
    Equivalently, the symmetrised Schur complement
    $S_s := B A_s^{-1} B^T = M^{1/2} K K^T M^{1/2}$ satisfies
    \begin{equation}
        \sigma_-^2 \, \|q\|_M^2 \;\leq\; q^T S_s\, q \;\leq\; \sigma_+^2 \, \|q\|_M^2 ,
        \label{eq:schur_equiv}
    \end{equation}
    i.e.\ the pressure mass matrix is spectrally equivalent to the Schur complement with
    $h$-independent constants.

    \emph{Step 2: the symmetric core.} Let $\mathcal{A}_s$ be the symmetric part from
    \eqref{eq:splitting} and $\mathcal{P}_s := \begin{bmatrix} A_s & B^T \\ & -M
        \end{bmatrix}$. Block back-substitution for $\mathcal{P}_s [\hat{\mathbf{u}}; \hat p]
        = [\mathbf{r}_u; \mathbf{r}_p]$ gives $\hat p = -M^{-1}\mathbf{r}_p$ and
    $\hat{\mathbf{u}} = A_s^{-1}(\mathbf{r}_u + B^T M^{-1} \mathbf{r}_p)$, whence
    \begin{equation*}
        \mathcal{A}_s \mathcal{P}_s^{-1} =
        \begin{bmatrix}
            I & 0 \\ B A_s^{-1} & S_s M^{-1}
        \end{bmatrix},
        \qquad
        X := D^{-1} \mathcal{A}_s \mathcal{P}_s^{-1} D =
        \begin{bmatrix}
            I & 0 \\ K & K K^T
        \end{bmatrix} .
    \end{equation*}
    The negative sign on the Schur block of $\mathcal{P}$ is essential here: it produces
    the positive-semidefinite block $K K^T$, which is in fact positive definite by
    \eqref{eq:schur_equiv}; with $+M$ in place of $-M$ the lower diagonal block would be
    $-K K^T$, the field of values would contain the origin, and no bound of the form
    \eqref{eq:fov} could hold. For $z = [z_1; z_2]$ set $w := K^T z_2$. Then, by the
    Young inequality $|\langle z_1, w\rangle| \leq \tfrac12(\|z_1\|^2 + \|w\|^2)$ and
    $\|w\| \geq \sigma_- \|z_2\|$,
    \begin{equation*}
        \langle X z, z \rangle = \|z_1\|^2 + \langle z_1, w \rangle + \|w\|^2
        \geq \tfrac12 \left( \|z_1\|^2 + \|w\|^2 \right)
        \geq \gamma_s \|z\|^2 ,
        \qquad
        \gamma_s := \tfrac12 \min\{ 1, \sigma_-^2 \} ,
    \end{equation*}
    while splitting $X$ into its three non-zero blocks gives
    $\|X\| \leq 1 + \sigma_+ + \sigma_+^2 =: \Gamma_s$. Thus \eqref{eq:fov} holds for the
    symmetric core $\mathcal{A}_s \mathcal{P}_s^{-1}$ with constants
    $(\gamma_s, \Gamma_s)$.

    \emph{Step 3: perturbation.} Let
    \begin{equation*}
        \mathcal{E} := \mathcal{A} - \mathcal{A}_s =
        \begin{bmatrix}
            A - A_s & B_1^T - B^T \\ B_2 - B & 0
        \end{bmatrix},
        \qquad
        \mathcal{F} := \mathcal{P} - \mathcal{P}_s =
        \begin{bmatrix}
            A - A_s & B_1^T - B^T \\ 0 & 0
        \end{bmatrix} .
    \end{equation*}
    By Lemma~\ref{lem:perturbation},
    $\mathbf{w}^T (A - A_s) \mathbf{v} = \tfrac12 \big( a(\mathbf{v},\mathbf{w}) -
        a(\mathbf{w},\mathbf{v}) \big) \leq \tfrac12 \varepsilon_h \|\mathbf{v}\|_a
        \|\mathbf{w}\|_a$, and $B_1 - B = -(B_2 - B) = \tfrac12 (B_1 - B_2)$ is bounded by
    $\tfrac12 \varepsilon_h$ in the same pair of norms, so \eqref{eq:norm_equiv} yields a
    constant $c_1 = c_1(C_a)$ with
    \begin{equation*}
        \|D^{-1} \mathcal{E} D^{-1}\| \leq c_1 \varepsilon_h ,
        \qquad
        \|D^{-1} \mathcal{F} D^{-1}\| \leq c_1 \varepsilon_h .
    \end{equation*}
    The matrix $\Pi := D^{-1} \mathcal{P}_s D^{-1} = \begin{bmatrix} I & K^T \\ & -I
        \end{bmatrix}$ satisfies $\Pi^2 = I$, hence
    $\|\Pi^{-1}\| = \|\Pi\| \leq 1 + \sigma_+$. Since
    $D^{-1} \mathcal{P} D^{-1} = \Pi + D^{-1}\mathcal{F}D^{-1}$, a Neumann series shows
    that whenever $c_1 (1+\sigma_+) \varepsilon_h \leq \tfrac12$, the preconditioner
    $\mathcal{P}$ is invertible with $\|D \mathcal{P}^{-1} D\| \leq 2(1+\sigma_+)$. From
    the identity $\mathcal{A}\mathcal{P}^{-1} - \mathcal{A}_s \mathcal{P}_s^{-1} =
        \mathcal{E}\mathcal{P}^{-1} - \mathcal{A}_s \mathcal{P}_s^{-1} \mathcal{F}
        \mathcal{P}^{-1}$ we obtain, in the transformed variables,
    \begin{equation*}
        \Delta := D^{-1} \big( \mathcal{A}\mathcal{P}^{-1} - \mathcal{A}_s\mathcal{P}_s^{-1} \big) D
        = \big( D^{-1}\mathcal{E}D^{-1} \big) \big( D\mathcal{P}^{-1}D \big)
        - X \big( D^{-1}\mathcal{F}D^{-1} \big) \big( D\mathcal{P}^{-1}D \big) ,
    \end{equation*}
    whence $\|\Delta\| \leq c_1 (1 + \Gamma_s)\, 2(1+\sigma_+)\, \varepsilon_h =: c_2
        \varepsilon_h$. The Euclidean field of values is stable under such perturbations:
    \begin{equation*}
        \langle (X + \Delta) z, z \rangle \geq (\gamma_s - \|\Delta\|) \|z\|^2 ,
        \qquad
        \|(X + \Delta) z\| \leq (\Gamma_s + \|\Delta\|) \|z\| .
    \end{equation*}
    Since $\varepsilon_h = C \hat h_{\tilde\Gamma_D}^{\zeta} \to 0$ as $h \to 0$, we may
    choose $h_0 \leq h_1$ such that $c_1(1+\sigma_+)\varepsilon_h \leq \tfrac12$ and
    $c_2 \varepsilon_h \leq \gamma_s/2$ for all $h \leq h_0$; then \eqref{eq:fov} holds
    with $\gamma := \gamma_s/2$ and $\Gamma := \Gamma_s + \gamma_s/2$, which depend only
    on $C_a, C_A, \beta, C_b$. Finally, the residual reduction factor of GMRES follows
    from the bound of Eisenstat, Elman and Schultz \cite{eisenstat1983variational}, which
    holds verbatim in any inner product \cite{starke1997fov}.
\end{proof}

\begin{remark}[choice of inner product]
    \label{rem:innerproduct}
    Theorem~\ref{thm:fov} bounds the convergence of GMRES implemented with right
    preconditioning in the $\mathcal{H}$-inner product, which measures residuals in the
    natural dual norm $\|\mathbf{r}_u\|_{A_s^{-1}}^2 + \|\mathbf{r}_p\|_{M^{-1}}^2$.
    Evaluating this inner product requires one solve with $A_s$ and one with $M$ per
    orthogonalisation step; since the preconditioner already performs exact solves with
    $A$ and $M$, this is realisable at comparable cost. If, as is common in practice, GMRES is
    run in the Euclidean inner product instead, the residual bound transfers with a
    norm-equivalence factor:
    \begin{equation*}
        \|r_k\|_2 \;\leq\; \kappa(\mathcal{H})^{1/2}
        \left(1 - \gamma^2/\Gamma^2\right)^{k/2} \|r_0\|_2 ,
    \end{equation*}
    since the Euclidean-optimal GMRES polynomial is at least as good as the
    $\mathcal{H}$-optimal one, up to the equivalence of the two norms. The factor
    $\kappa(\mathcal{H})^{1/2} = O(h^{-1})$ is not mesh-independent, but it enters only as
    a multiplicative constant, not in the contraction rate: the number of iterations
    needed to reach a fixed relative tolerance grows at most by an additive
    $O(\log(1/h))$ term.
\end{remark}

\begin{corollary}[inexact velocity block]
    \label{cor:inexact}
    Let the assumptions of Theorem~\ref{thm:fov} hold, and let $A_0$ be an invertible,
    possibly non-symmetric, approximation of $A$ whose error propagation operator is a
    uniform contraction in the energy norm: there exists an $h$-independent
    $\rho < \tfrac12$ such that
    \begin{equation}
        \| v - A_0^{-1} A\, v \|_{A_s} \;\leq\; \rho \, \|v\|_{A_s}
        \qquad \forall v .
        \label{eq:inexact_A}
    \end{equation}
    Then the conclusions of Theorem~\ref{thm:fov} hold for the preconditioner
    $$\mathcal{P}_0 = \begin{bmatrix} A_0 & B_1^T \\ & -M \end{bmatrix},$$
    with constants $\gamma, \Gamma$ depending additionally only on $\rho$.
    In particular, a multigrid V-cycle for the SBM velocity block whose energy-norm
    contraction number is bounded by $\rho < \tfrac12$ uniformly in $h$
    preserves mesh-independent GMRES convergence of the outer solver.
\end{corollary}

\begin{proof}
    We follow the proof of Theorem~\ref{thm:fov}, indicating only the changes; all
    notation is as introduced there. By Lemma~\ref{lem:perturbation} and
    \eqref{eq:norm_equiv}, $\hat H := A_s^{-1/2} A A_s^{-1/2} = I + N_A$ with
    $\|N_A\| \leq \varepsilon_h / C_a$. Set $G := A_s^{1/2} A_0^{-1} A_s^{1/2}$;
    assumption \eqref{eq:inexact_A} states exactly $\|I - G \hat H\| \leq \rho$, whence
    $\|G\| \leq (1+\rho) \|\hat H^{-1}\| \leq (1+\rho)(1 - \varepsilon_h/C_a)^{-1}$ and
    \begin{equation*}
        \|I - G\| \leq \|I - G \hat H\| + \|G\| \, \|\hat H - I\|
        \leq \rho + C \varepsilon_h =: \tilde\rho .
    \end{equation*}
    Decreasing $h_0$ if necessary, $\tilde\rho < \tfrac12$ for all $h \leq h_0$; in
    particular $G$ is invertible with $\|G^{-1}\| \leq (1 - \tilde\rho)^{-1} \leq 2$.

    Define the symmetric-core preconditioner $\mathcal{P}_{0,s} := \begin{bmatrix} A_0 & B^T \\ & -M \end{bmatrix}$ (only the
    coupling block is symmetrised). Back-substitution as in Step 2 of
    Theorem~\ref{thm:fov} gives
    \begin{equation*}
        X_0 := D^{-1} \mathcal{A}_s \mathcal{P}_{0,s}^{-1} D =
        \begin{bmatrix}
            G & (G - I) K^T \\ K G & K G K^T
        \end{bmatrix} .
    \end{equation*}
    For $z = [z_1; z_2]$ set $w := K^T z_2$ and $s := z_1 + w$, so that
    $X_0 z = [G s - w ;\; K G s]$. Writing $G = I + N$ with $\|N\| \leq \tilde\rho$ and
    using the Young inequality,
    \begin{equation*}
        \langle X_0 z, z \rangle
        = \langle G s, s \rangle - \langle w, z_1 \rangle
        = \|s\|^2 + \langle N s, s \rangle - \langle w, s \rangle + \|w\|^2
        \geq \big( \tfrac12 - \tilde\rho \big) \|s\|^2 + \tfrac12 \|w\|^2
        \geq \big( \tfrac12 - \tilde\rho \big) \big( \|s\|^2 + \|w\|^2 \big) .
    \end{equation*}
    Since $\|z_1\|^2 \leq 2(\|s\|^2 + \|w\|^2)$ and
    $\|z_2\|^2 \leq \sigma_-^{-2} \|w\|^2$, we conclude
    \begin{equation*}
        \langle X_0 z, z \rangle \geq \gamma_0 \|z\|^2 ,
        \qquad
        \gamma_0 := \frac{\tfrac12 - \tilde\rho}{2 + \sigma_-^{-2}} > 0 ,
    \end{equation*}
    while the blockwise bound gives
    $\|X_0\| \leq \|G\| (1 + \sigma_+ + \sigma_+^2) + \tilde\rho\, \sigma_+
        \leq (1 + \tilde\rho)(1 + \sigma_+ + \sigma_+^2) + \tilde\rho\, \sigma_+ =: \Gamma_0$.

    Step 3 of the proof of Theorem~\ref{thm:fov} carries over verbatim with $(\mathcal{P}, \mathcal{P}_s)$ replaced by
    $(\mathcal{P}_0, \mathcal{P}_{0,s})$: the perturbation $$\mathcal{F}_0 := \mathcal{P}_0 - \mathcal{P}_{0,s} = \begin{bmatrix}
            0 & B_1^T - B^T \\ 0 & 0 \end{bmatrix}$$
    satisfies
    $\|D^{-1} \mathcal{F}_0 D^{-1}\| \leq c_1 \varepsilon_h$, and
    $$\Pi_0 := D^{-1} \mathcal{P}_{0,s} D^{-1} = \begin{bmatrix} G^{-1} & K^T \\ & -I
        \end{bmatrix}$$
    is invertible with
    $$\Pi_0^{-1} = \begin{bmatrix} G & G K^T \\ & -I \end{bmatrix},$$
    so that
    $\|\Pi_0^{-1}\| \leq (1+\tilde\rho)(1+\sigma_+) + 1$. The same Neumann-series and
    perturbation arguments then yield, after a further decrease of $h_0$ if necessary,
    the bounds \eqref{eq:fov} for $\mathcal{A}\mathcal{P}_0^{-1}$ with
    $\gamma := \gamma_0/2$ and $\Gamma := \Gamma_0 + \gamma_0/2$, depending only on
    $C_a, C_A, \beta, C_b$ and $\rho$.
\end{proof}

\begin{remark}[pre-asymptotic regime]
    \label{rem:coarse}
    Theorem~\ref{thm:fov} is conditional on $h \leq h_0$, i.e.\ on the mesh resolving the
    geometry in the sense of \eqref{eq:ass_d}. On meshes that are too coarse relative to
    the embedded geometry (in particular when the surrogate normal is nearly orthogonal
    to the true normal, so that $\inf (\mathbf{n} \cdot \tilde{\mathbf{n}})$ degrades and
    $\|\mathbf{d}\|/h = O(1)$), both the coercivity constants of
    \cite{atallah2020analysis} and the perturbation bound of Lemma~\ref{lem:perturbation}
    deteriorate, and the theory gives no mesh-independent bound. The theory thus only
    guarantees uniform convergence once the mesh resolves the embedded geometry; it makes
    no prediction about solver behaviour in the under-resolved regime.
\end{remark}

\section{Implementation details}
\label{sec:implementation}

The numerical implementation is based on the open-source finite element library \texttt{deal.II}~\cite{dealii2024}. It
provides a comprehensive framework for the implementation of finite element methods, including mesh handling, finite
element spaces, assembly of linear systems, and interfaces to various linear algebra solvers and preconditioners. The
treatment of the embedded geometry and the assembly of the SBM bilinear form follow the implementation developed for
the geometric multigrid solver of~\cite{wichrowski2025geometric, shypatch}; we summarise the components relevant to the
Stokes saddle-point system below.

While in this work the velocity block $A$ is solved using a direct solver, the geometric multigrid preconditioner
introduced in~\cite{wichrowski2025geometric} and extended in~\cite{shypatch} can be used to invert the velocity block
efficiently, and is expected to deliver $h$-independent convergence rates for higher-degree polynomials.

\subsection{Background mesh and geometry handling}
When implementing SBM on a non-body-fitted mesh, we need to handle cells intersected by the true boundary $\Gamma$. We
use a level set function $\phi(\mathbf{x})$ to implicitly define the domain $\Omega = \{\mathbf{x} \mid
    \phi(\mathbf{x}) < 0\}$, with $\Gamma = \{\mathbf{x} \mid \phi(\mathbf{x}) = 0\}$. Cells of the background mesh are
classified based on their intersection with the zero level set: cells entirely inside $\Omega$ (interior), cells
entirely outside $\Omega$ (exterior), and cells intersected by $\Gamma$. Of these, only the interior cells are
classified as active. In this specific implementation of the Stokes problem, all the geometric complexity lies in the
interior of the domain: the outer boundary is fitted to the mesh, while the embedded inner boundary is handled by SBM.

In our approach, degrees of freedom are formally assigned to all cells of the background mesh, including those
classified as non-active. This strategy simplifies the implementation by maintaining a consistent sparsity pattern and
avoiding the need to frequently re-compute the global matrix graph, which is particularly advantageous for problems
with moving boundaries. While it introduces zero rows and columns into the global system matrix for degrees of freedom
associated only with inactive cells, this does not adversely affect the convergence of the iterative GMRES solver, as
the corresponding entries in both the matrix and the right-hand side vanish. To ensure invertibility for the direct
solver applied to the velocity block, these inactive degrees of freedom are constrained, effectively placing $1$ on the
diagonal of each such row. In an optimised production implementation, the computational overhead of these inactive
regions could be minimised by restricting mesh refinement strictly to the vicinity of the active domain.

\subsection{Processing surrogate boundary and matrix assembly}
\label{sec:sbm_assembly}
The SBM requires computing the closest-point projection from points on the surrogate boundary $\Gamma\pr$ to the
true boundary $\Gamma$. In the geometries considered here the inner boundary is a circle or a composition of circles, so
the projection reduces to a radial projection onto the closest circle for a given point. For general geometries this
projection is found by solving a local nonlinear optimisation problem for each quadrature point on $\Gamma\pr$,
minimising the distance to the zero level set of $\phi(\mathbf{x})$ via a Lagrange-multiplier formulation solved with a
Newton--Raphson method, as described in~\cite{wichrowski2025geometric}; since this closest-point search on a level set
is natively supported in \texttt{deal.II}, generalising the present implementation to more complex geometries is
straightforward. The search for the closest point is performed only on the faces of active cells adjacent to
intersected cells. While this does not guarantee that the closest point is found inside the cell, the extrapolation is
expected to yield a good approximation of the boundary condition even when it lies slightly outside.

The extension of the function values from the surrogate boundary to the true boundary is implemented by evaluating the
shape functions of the active elements at the projected points on $\Gamma$. Since the finite element basis functions
are polynomials over each element, this evaluation is performed by mapping the physical coordinates of the points on
$\Gamma$ to the reference unit cell of the respective element. This effectively evaluates the shape functions at points
that may lie slightly outside the standard $[0,1]^d$ range of the unit cell, naturally providing the necessary
extrapolation $\mathbf{u}_p$ of the velocity values without requiring the explicit computation of high-order
derivatives. For the extrapolation $\mathbf{v}_p$ of the test functions we use a first-order Taylor expansion, as it
provides a better error estimate~\cite{atallah2022high}. The assembly of the cell contributions and interior faces to
the system matrix and right-hand side vector is then performed using the standard finite element assembly process
provided by \texttt{deal.II}.

\subsection{Block preconditioner}

The application of the block upper-triangular preconditioner \eqref{eq:prec_used} amounts to computing
\begin{equation}
    \begin{bmatrix}
        \mathbf{u} \\
        \mathbf{p}
    \end{bmatrix}
    =
    \begin{bmatrix}
        A & B^T \\
          & -M
    \end{bmatrix}^{-1}
    \begin{bmatrix}
        \mathbf{r}_u \\
        \mathbf{r}_p
    \end{bmatrix} ,
\end{equation}
where $A$ is the velocity block of the system matrix, $B^T$ is the transpose of the divergence block, $M$ is the
pressure mass matrix approximating the Schur complement (with the negative sign required by the field-of-values
analysis of Section~\ref{sec:theory}), and $\mathbf{r}_u$, $\mathbf{r}_p$ are the velocity and
pressure components of the residual. By block back-substitution, this is realised as
\begin{align}
    \mathbf{p} & = -M^{-1} \mathbf{r}_p ,                   \\
    \mathbf{u} & = A^{-1} (\mathbf{r}_u - B^T \mathbf{p}) .
\end{align}
The preconditioner is applied within each GMRES iteration; both $A^{-1}$ and $M^{-1}$ are realised by direct solvers
in the present implementation. The effectiveness of the preconditioner is assessed by the number of iterations
required for convergence, reported in Section~\ref{sec:results}.

\section{Numerical results}
\label{sec:results}

In this section, we examine the performance of the block upper-triangular preconditioner \eqref{eq:prec_used} applied
to the SBM discretisation of the Stokes problem. We investigate if the mesh-independent convergence established in
Section~\ref{sec:theory} is realised in practice. Theorem~\ref{thm:fov} predicts that, once the mesh resolves the
embedded geometry, preconditioned GMRES converges at a rate bounded independently of the mesh size. However, the theory
is silent about the under-resolved regime (Remark~\ref{rem:coarse}), and configurations with closely spaced embedded
boundaries are likely to enter this regime on coarse meshes. Our experiments are designed to probe both aspects: we
measure iteration counts under uniform mesh refinement for two embedded geometries of increasing difficulty, comparing
discontinuous- and continuous-pressure discretisations and the non-symmetric and symmetric formulations of the velocity
gradient introduced in Section~\ref{sec:sbm_stokes}.

The remainder of this section is organised as follows. Section~\ref{sec:exp_setup} describes the test geometries,
boundary conditions, and solver settings. Section~\ref{sec:iteration_counts} presents the iteration counts under mesh
refinement and discusses their dependence on the pressure space, the gradient formulation, and the geometry.
Section~\ref{sec:inexact_prec} examines the effect of replacing the exact velocity solve with an approximate multigrid
preconditioner.

\subsection{Experimental setup}
\label{sec:exp_setup}
We assess the solver on a family of test problems of increasing geometric complexity. The computational domain is the
unit square with body-fitted outer walls; the embedded inner boundaries are handled by the SBM. We consider two
geometric configurations, referred to as \emph{arrangements}:
\begin{itemize}
    \item \textbf{Arrangement 1}: a single embedded circular boundary;
    \item \textbf{Arrangement 2}: three closely spaced circular boundaries separated by narrow gaps.
\end{itemize}
The two arrangements are shown in Figure~\ref{fig:stream_tracer_comparison}.
In both cases the top and bottom walls carry Dirichlet conditions, the left and right walls carry homogeneous natural
(Neumann) conditions so that the pressure is not over-constrained, and the embedded SBM boundaries carry homogeneous
Dirichlet (no-slip) conditions, $\mathbf{u} = \mathbf{0}$. Each arrangement is solved on a sequence of uniformly
refined meshes with the block upper-triangular preconditioner \eqref{eq:prec_used}. For each velocity degree we compare
a discontinuous-pressure pair (denoted DG) with a continuous-pressure, Taylor--Hood--type pair (denoted $Q$), and for
each pair we report results for both the non-symmetric and the symmetric formulation of the velocity gradient. The
Nitsche penalty multiplier is fixed at $5$. GMRES is run without restarts to a fixed relative residual tolerance,
applying the velocity block exactly and the pressure Schur complement through the mass matrix $-M$ as in
\eqref{eq:prec_used}.

\begin{figure}[h]
    \centering
    \includegraphics[width=0.45\textwidth]{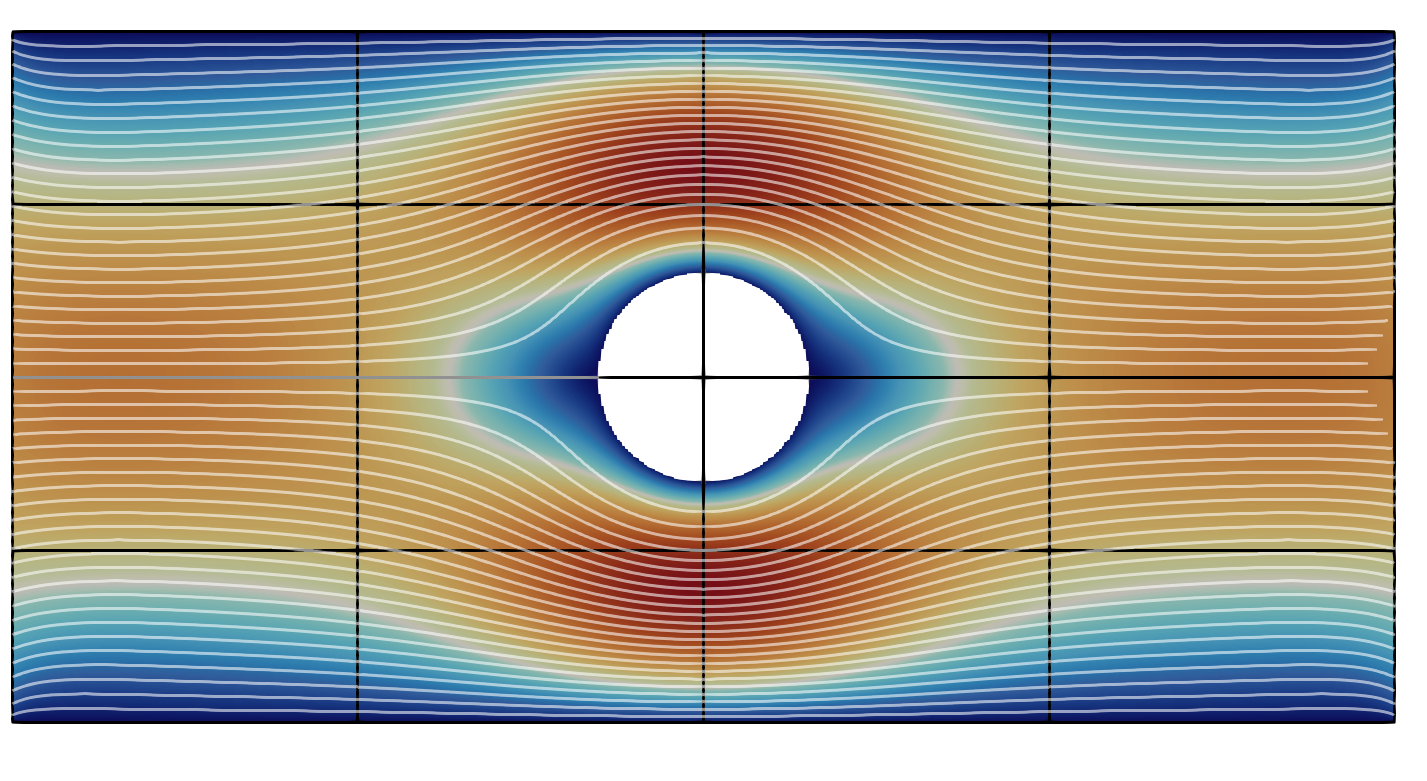}
    \hfill
    \includegraphics[width=0.45\textwidth]{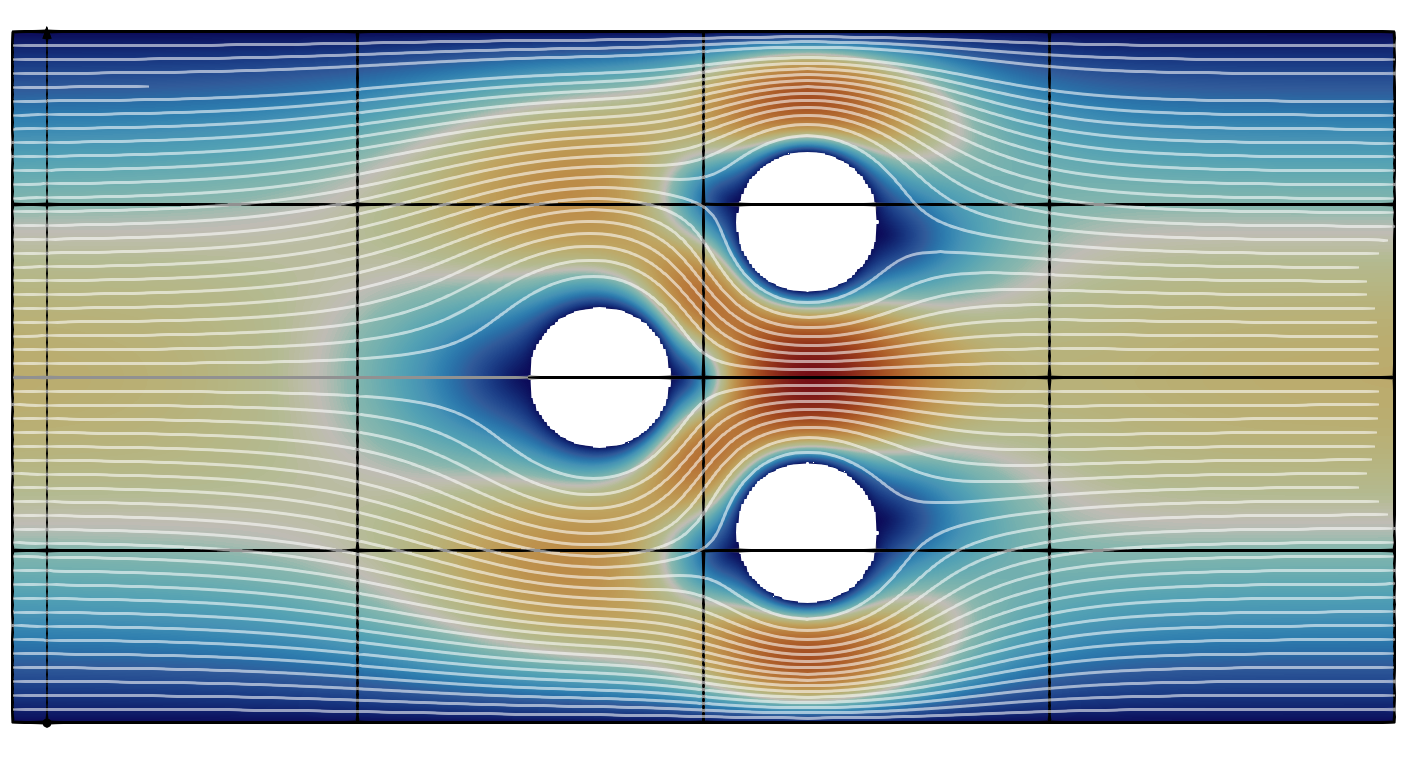}
    \caption{The two arrangements and the solution obtained with $\pairDGtwo$. Colour indicates the velocity
        magnitude and white lines are stream tracers, both taken from the finest refinement level; the mesh drawn on top is
        the coarsest grid used ($2$ refinements), shown for clarity, as overlaying the finest grid would obscure the flow field.
        Arrangement 2 exhibits the most intricate flow pattern owing to the
        narrow channels between the embedded boundaries.}
    \label{fig:stream_tracer_comparison}
\end{figure}

The arrangements are ordered by the geometric difficulty they pose: Arrangement 1 is a single obstacle, whereas
Arrangement 2 is more demanding, as the narrow gaps between its boundaries stress the geometric resolution assumption
\eqref{eq:ass_d} on coarse meshes.

Table~\ref{tab:dofs_deg2} reports the number of Stokes degrees of freedom at each refinement level for the two
arrangements and both discretisation pairs ($k=2$ velocity).

\begin{table}[h]
    \centering
    \begin{tabular}{cc|cccccccc}
      \multicolumn{1}{c}{} & \multicolumn{1}{c}{pair} & \multicolumn{1}{c}{2} & \multicolumn{1}{c}{3} & \multicolumn{1}{c}{4} & \multicolumn{1}{c}{5} & \multicolumn{1}{c}{6} & \multicolumn{1}{c}{7} & \multicolumn{1}{c}{8} & \multicolumn{1}{c}{9} \\
      \hline
      \multirow{2}{*}{Arr.\ 1} & $\Qcont{2}/\Qdg{0}$ & 156 & 592 & 2,292 & 9,068 & 35,948 & 142,960 & 570,456 & 2,278,892 \\
      & $\Qcont{2}/\Qcont{1}$ & 168 & 614 & 2,334 & 9,148 & 36,106 & 143,276 & 571,086 & 2,280,148 \\
      \cline{2-10}
      \multirow{2}{*}{Arr.\ 2} & $\Qcont{2}/\Qdg{0}$ & 142 & 558 & 2,250 & 8,886 & 35,368 & 141,080 & 563,034 & 2,250,036 \\
      & $\Qcont{2}/\Qcont{1}$ & 156 & 585 & 2,303 & 8,987 & 35,560 & 141,456 & 563,783 & 2,251,529 \\
      \hline
    \end{tabular}

    \caption{Number of Stokes degrees of freedom for $k=2$ at each refinement level.}
    \label{tab:dofs_deg2}
\end{table}

\subsection{Iteration counts under mesh refinement}
\label{sec:iteration_counts}
We now turn to the main quantity of interest: the number of preconditioned GMRES iterations required to reach the
prescribed tolerance as the mesh is refined. Tables~\ref{tab:iters_deg2} and~\ref{tab:iters_deg3}
report the iteration counts at each refinement level for the $\pairDGtwo$ and $\pairDGthree$ element pairs, respectively.

\begin{table}[h]
    \centering
    \begin{minipage}[t]{0.48\textwidth}
        \centering
    \begin{tabular}{cc|ccccccc}
      \multicolumn{9}{c}{\textbf{regular gradient}} \\
      \multicolumn{1}{c}{} & \multicolumn{1}{c}{$Q_h$} & \multicolumn{1}{c}{3} & \multicolumn{1}{c}{4} & \multicolumn{1}{c}{5} & \multicolumn{1}{c}{6} & \multicolumn{1}{c}{7} & \multicolumn{1}{c}{8} & \multicolumn{1}{c}{9} \\
      \hline
      \multirow{2}{*}{Arr.\ 1} & $\Qdg{0}$ & 13 & 14 & 15 & 16 & 16 & 17 & 16 \\
      & $\Qcont{1}$ & 14 & 17 & 17 & 20 & 23 & 23 & 24 \\
      \cline{2-9}
      \multirow{2}{*}{Arr.\ 2} & $\Qdg{0}$ & 23 & 47 & 26 & 26 & 28 & 28 & 31 \\
      & $\Qcont{1}$ & 30 & 113 & 49 & 74 & 57 & 70 & 49 \\
      \hline
    \end{tabular}

    \end{minipage}
    \hfill
    \begin{minipage}[t]{0.48\textwidth}
        \centering
    \begin{tabular}{cc|ccccccc}
      \multicolumn{9}{c}{\textbf{symmetric gradient}} \\
      \multicolumn{1}{c}{} & \multicolumn{1}{c}{$Q_h$} & \multicolumn{1}{c}{3} & \multicolumn{1}{c}{4} & \multicolumn{1}{c}{5} & \multicolumn{1}{c}{6} & \multicolumn{1}{c}{7} & \multicolumn{1}{c}{8} & \multicolumn{1}{c}{9} \\
      \hline
      \multirow{2}{*}{Arr.\ 1} & $\Qdg{0}$ & 14 & 15 & 16 & 16 & 16 & 16 & 16 \\
      & $\Qcont{1}$ & 15 & 18 & 16 & 19 & 22 & 23 & 21 \\
      \cline{2-9}
      \multirow{2}{*}{Arr.\ 2} & $\Qdg{0}$ & 25 & 32 & 24 & 24 & 24 & 23 & 23 \\
      & $\Qcont{1}$ & 30 & 42 & 47 & 54 & 43 & 34 & 30 \\
      \hline
    \end{tabular}

    \end{minipage}
    \caption{GMRES iterations for $\pairDGtwo$ elements across arrangement configurations and discretisation choices. NC = failed to reach the tolerance within the iteration cap.}
    \label{tab:iters_deg2}
\end{table}

\begin{table}[h]
    \centering
    \begin{minipage}[t]{0.48\textwidth}
        \centering
    \begin{tabular}{cc|cccccc}
      \multicolumn{8}{c}{\textbf{regular gradient}} \\
      \multicolumn{1}{c}{} & \multicolumn{1}{c}{$Q_h$} & \multicolumn{1}{c}{2} & \multicolumn{1}{c}{3} & \multicolumn{1}{c}{4} & \multicolumn{1}{c}{5} & \multicolumn{1}{c}{6} & \multicolumn{1}{c}{7} \\
      \hline
      \multirow{2}{*}{Arr.\ 1} & $\Qdg{1}$ & 15 & 19 & 19 & 30 & 18 & 23 \\
      & $\Qcont{2}$ & 21 & 21 & 19 & 53 & 30 & 53 \\
      \cline{2-8}
      \multirow{2}{*}{Arr.\ 2} & $\Qdg{1}$ & 20 & 296 & 116 & 204 & 164 & 72 \\
      & $\Qcont{2}$ & 26 & NC & NC & NC & NC & NC \\
      \hline
    \end{tabular}

    \end{minipage}
    \hfill
    \begin{minipage}[t]{0.48\textwidth}
        \centering
    \begin{tabular}{cc|cccccc}
      \multicolumn{8}{c}{\textbf{symmetric gradient}} \\
      \multicolumn{1}{c}{} & \multicolumn{1}{c}{$Q_h$} & \multicolumn{1}{c}{2} & \multicolumn{1}{c}{3} & \multicolumn{1}{c}{4} & \multicolumn{1}{c}{5} & \multicolumn{1}{c}{6} & \multicolumn{1}{c}{7} \\
      \hline
      \multirow{2}{*}{Arr.\ 1} & $\Qdg{1}$ & 20 & 16 & 18 & 23 & 18 & 24 \\
      & $\Qcont{2}$ & 20 & 17 & 21 & 27 & 22 & 56 \\
      \cline{2-8}
      \multirow{2}{*}{Arr.\ 2} & $\Qdg{1}$ & 21 & NC & 89 & 208 & 279 & 111 \\
      & $\Qcont{2}$ & 30 & NC & 207 & 323 & NC & NC \\
      \hline
    \end{tabular}

    \end{minipage}
    \caption{GMRES iterations for $\pairDGthree$ elements across arrangement configurations and discretisation choices. NC = failed to reach the tolerance within the iteration cap.}
    \label{tab:iters_deg3}
\end{table}

\paragraph{$\pairDGtwo$ results.}
For the lower-order pair the clearest trend is the contrast between the two pressure spaces. For the
discontinuous-pressure ($\Qdg{0}$) pairs the iteration counts are low and flat across all refinement levels. On
Arrangement~1 they remain in the range $13$--$17$ (non-symmetric) and $14$--$16$ (symmetric), confirming the
mesh-independent behaviour predicted by Theorem~\ref{thm:fov} and the effectiveness of the pressure mass matrix as a
Schur complement approximation. On the more challenging Arrangement~2 there is an iteration spike on the coarsest mesh
(up to $47$ non-symmetric, $32$ symmetric) caused by the unresolved narrow gaps, after which the counts return to the
range $26$--$31$ (non-symmetric) and $23$--$25$ (symmetric) as the geometry is resolved.

The continuous-pressure ($\Qcont{1}$) pairs are more sensitive to the geometry. On Arrangement~1 the counts grow only
mildly with refinement, from $14$ to $24$ (non-symmetric) and from $15$ to a peak of $23$ (symmetric). On Arrangement~2
they are substantially higher and depend non-monotonically on the mesh: the non-symmetric formulation reaches $113$ at
refinement level~4 and remains variable ($49$--$74$) at finer levels, whereas the symmetric formulation peaks at $54$
and decreases steadily, returning to $30$ on the finest mesh. The symmetric-gradient formulation thus provides a
meaningful practical benefit for the continuous-pressure pair in the under-resolved regime, which we attribute to the
better-conditioned velocity block it produces. A tangential-derivative penalisation has been proposed in the literature
to further stabilise the SBM in under-resolved regimes~\cite{SBMP1}; it is not included in our implementation.

\paragraph{$\pairDGthree$ results.}
The higher-order pair ($\pairDGthree$) reveals a different behaviour. On Arrangement~1 both pressure spaces still
converge, but the iteration counts are less uniform across refinement levels: the DG pair oscillates between $15$ and
$30$ (non-symmetric) and $16$ and $24$ (symmetric), while the continuous-pressure pair reaches up to $53$
(non-symmetric) and $56$ (symmetric). On Arrangement~2 the situation deteriorates. The DG pair reaches counts as high
as $296$ (non-symmetric) and $279$ (symmetric) before recovering to $72$ and $111$ respectively at the finest level,
indicating that the geometric resolution assumption is stressed over a much wider range of mesh sizes at higher
polynomial degree. The continuous-pressure pair fails to converge (NC) at almost all refinement levels for
Arrangement~2; the symmetric formulation recovers only at intermediate levels (with counts as high as $323$) before
failing again on the two finest meshes. These results suggest that the block preconditioner, while effective for $k=2$,
requires finer meshes to reach the resolved regime for $k=3$, and that additional stabilisation may be necessary for
high-order discretisations in geometrically demanding configurations.

\subsection{Inexact preconditioners}
\label{sec:inexact_prec}

All results in Section~\ref{sec:iteration_counts} rely on exact solves with the velocity block $A$. In practice, large
three-dimensional problems demand an approximate inverse. A natural candidate is the geometric multigrid preconditioner
of~\cite{wichrowski2025geometric,shypatch}: at its core lies the \emph{Full-Residual Shy Patch} smoother, a subspace
correction strategy that forms overlapping vertex patches. The shy patch smoother was shown in~\cite{shypatch} to
deliver low and mesh-stable iteration counts for the SBM--Poisson problem up to polynomial degree $p=3$ in three
dimensions.

Table~\ref{tab:iters_mg_deg2} reports outer iteration counts for $\pairDGtwo$ elements on Arrangement~1 when $A^{-1}$
is replaced by an inexact solve: $N_{\mathrm{inner}}$ steps of GMRES preconditioned by one V-cycle of the shy-patch
multigrid preconditioner are applied to the velocity block, with $N_{\mathrm{inner}} \in \{2, 5\}$. Because this inner
solver is itself iterative, the block preconditioner becomes \emph{non-linear}, and the outer Krylov method must be
switched from GMRES to FGMRES~\cite{Saad1993} to accommodate the variable preconditioner. We use restarted FGMRES with
a restart length of $30$.

\begin{table}[h]
    \centering
    \begin{minipage}[t]{0.48\textwidth}
        \centering
    \begin{tabular}{cc|ccccc}
      \multicolumn{7}{c}{\textbf{regular gradient}} \\
      \multicolumn{1}{c}{} & \multicolumn{1}{c}{$Q_h$} & \multicolumn{1}{c}{5} & \multicolumn{1}{c}{6} & \multicolumn{1}{c}{7} & \multicolumn{1}{c}{8} & \multicolumn{1}{c}{9} \\
      \hline
      \multirow{2}{*}{$N_{\mathrm{inner}} = 2$} & $\Qdg{0}$ & 16 & NC & NC & - & - \\
      & $\Qcont{1}$ & 19 & NC & NC & - & - \\
      \cline{2-7}
      \multirow{2}{*}{$N_{\mathrm{inner}} = 5$} & $\Qdg{0}$ & 16 & 17 & 18 & 19 & NC \\
      & $\Qcont{1}$ & 18 & 21 & 26 & 26 & NC \\
      \hline
    \end{tabular}

    \end{minipage}
    \hfill
    \begin{minipage}[t]{0.48\textwidth}
        \centering
    \begin{tabular}{cc|ccccc}
      \multicolumn{7}{c}{\textbf{symmetric gradient}} \\
      \multicolumn{1}{c}{} & \multicolumn{1}{c}{$Q_h$} & \multicolumn{1}{c}{5} & \multicolumn{1}{c}{6} & \multicolumn{1}{c}{7} & \multicolumn{1}{c}{8} & \multicolumn{1}{c}{9} \\
      \hline
      \multirow{2}{*}{$N_{\mathrm{inner}} = 2$} & $\Qdg{0}$ & NC & 69 & NC & NC & - \\
      & $\Qcont{1}$ & NC & NC & NC & - & - \\
      \cline{2-7}
      \multirow{2}{*}{$N_{\mathrm{inner}} = 5$} & $\Qdg{0}$ & 17 & 17 & 22 & 18 & NC \\
      & $\Qcont{1}$ & 18 & 21 & 30 & 25 & NC \\
      \hline
    \end{tabular}

    \end{minipage}
    \caption{FGMRES iterations for $\pairDGtwo$ elements with shy-patch multigrid preconditioning of the velocity block
        (Arrangement~1 only).  $N_{\mathrm{inner}}$ denotes the number of inner GMRES steps applied to the velocity block.
        NC = failed, $-$ = not run.}
    \label{tab:iters_mg_deg2}
\end{table}

The results reveal a strong sensitivity to the number of inner iterations. With only $N_{\mathrm{inner}}=2$ inner steps
the method frequently fails to converge (NC), indicating that the approximate velocity solve is too inaccurate to keep
the outer iteration bounded. Increasing to $N_{\mathrm{inner}}=5$ restores convergence at most refinement levels; the
outer iteration counts are then close to those of the exact-solve experiments. The counts grow by a small but
approximately fixed increment per refinement level (e.g.\ $16\to17\to18\to19$ for the $\Qdg{0}$ pair with non-symmetric
gradient), but this mild growth is not introduced by the inexact inner solve: the exact-solver counts show the same
pattern in the same refinement range (cf.\ Table~\ref{tab:iters_deg2}). It is a property of the problem itself (most
likely the gradual resolution of the embedded geometry) rather than a sign that the multigrid preconditioner degrades
with mesh refinement. The symmetric-gradient formulation with $N_{\mathrm{inner}}=5$ shows the same trend.

The contrast between $N_{\mathrm{inner}}=2$ and $N_{\mathrm{inner}}=5$ is consistent with the analysis, though it
should be read with care. Corollary~\ref{cor:inexact} provides a \emph{sufficient} condition: if the inner solve
contracts the velocity error in the energy norm with a factor $\rho < \tfrac12$, the field-of-values bounds hold and
the outer iteration is mesh-independent. The bound itself degrades \emph{gradually} as $\rho$ increases towards
$\tfrac12$ (the coercivity constant $\gamma_0 = (\tfrac12 - \tilde\rho)/(2 + \sigma_-^{-2})$ vanishes continuously at
the threshold), and the value $\tfrac12$ is an artefact of the proof technique rather than a physical limit. Once this
sufficient condition is violated the corollary offers no guarantee; it does not predict failure. Two further caveats
apply. First, \eqref{eq:inexact_A} assumes a \emph{fixed, linear} approximate inverse $A_0$, whereas
$N_{\mathrm{inner}}$ steps of GMRES form a variable, non-linear preconditioner, so the corollary is suggestive here
rather than directly applicable. Second, the non-convergence (NC) entries denote failure to reach the tolerance within
the iteration cap, not divergence.

We emphasise that this sensitivity has no counterpart in body-fitted Stokes solvers, where block-triangular
preconditioning with a few inner multigrid sweeps on the velocity block is routinely robust. The distinguishing feature
of the SBM system is the non-normality of the velocity block: as observed in Section~\ref{sec:theory}, every
non-symmetric term carries a factor of the distance vector $\mathbf{d}$ and is concentrated on the surrogate boundary.
This matters for the inexact solve in a specific way. The inner GMRES iteration monitors and reduces the
\emph{residual} of the velocity system, whereas the hypothesis of Corollary~\ref{cor:inexact} concerns the contraction
of the \emph{error in the energy norm}. For a nearly symmetric block these two quantities are equivalent with moderate
constants, but for a non-normal block the equivalence degrades, and a few inner steps may achieve a respectable
residual reduction while contracting the energy error only weakly. We therefore read the contrast between the two row
groups of Table~\ref{tab:iters_mg_deg2} as follows: with $N_{\mathrm{inner}}=2$ the effective contraction $\rho$
remains too large for the field-of-values argument to apply, whereas with $N_{\mathrm{inner}}=5$ it falls into the
regime covered by Corollary~\ref{cor:inexact} and near-mesh-independent convergence is restored. A quantitative study
of the inner contraction in the energy norm is left for future work.

The behaviour under restarting supports this reading. Restarting discards the accumulated Krylov subspace, and for a
slowly converging problem this loss is severe. In additional experiments with a restart length of $100$, several of the
failing $N_{\mathrm{inner}}=2$ configurations made steady progress up to the first restart, after which convergence
visibly deteriorated. The NC entries should therefore be understood as slow, restart-limited convergence rather than
divergence: a sufficiently accurate inner solve lets the outer iteration converge well within a single restart cycle,
whereas a marginal one pushes it into a regime where restarted FGMRES is no longer effective.

In summary, the block upper-triangular preconditioner~\eqref{eq:prec_used} with an inexact velocity solve is
\emph{highly sensitive to the quality of that solve}, in a way not seen for body-fitted discretisations. A sufficient
number of inner iterations, combined with a restart length long enough for the outer iteration to converge within one
cycle, is essential to preserve the mesh-independent convergence established by the exact-solve theory.

\FloatBarrier
\section{Conclusion}
\label{sec:conclusion}
We investigated the performance of a block upper-triangular preconditioner for the Stokes problem discretised with the Shifted Boundary Method (SBM). Our experiments across two geometric arrangements demonstrated that the preconditioner provides stable iteration counts for discontinuous-pressure discretisations, in line with the field-of-values analysis of Section~\ref{sec:theory}. Continuous-pressure pairs are more sensitive in the presence of closely spaced embedded boundaries (Arrangement~2): for the lower-order pair the symmetric-gradient formulation markedly improves convergence on resolved meshes, whereas at higher polynomial degree the continuous-pressure pair remains unreliable on the demanding geometry, indicating that additional stabilisation is needed in the under-resolved regime. To the best of our knowledge, this is the first demonstration of scalable block preconditioning for saddle-point systems arising from SBM discretisations of the Stokes equations.

An important extension of the method presented here is to replace the exact velocity block solve with a high-quality
approximate inverse. Our experiments with the shy-patch geometric multigrid
preconditioner~\cite{wichrowski2025geometric,shypatch} (Section~\ref{sec:inexact_prec}) confirm that the block
preconditioner is highly sensitive to the accuracy of this inner solve: iteration counts increase substantially when
only a small number of V-cycles is used. Achieving mesh-independent outer iteration counts requires either more
multigrid smoothing sweeps or a more powerful smoother, and is an important target for future work. A fully scalable
solver of this form would make the method practical for large-scale 3D simulations.

\section*{Declarations}
Language models (Claude, Gemini) were used in the drafting process to accelerate convergence toward a readable manuscript. The authors retain full accountability for all scientific content.

\FloatBarrier
\bibliographystyle{siam}
\bibliography{stokes_literature}

@article{SBMP1,
  title = {The shifted boundary method for embedded domain computations. Part I: Poisson and Stokes problems},
  journal = {Journal of Computational Physics},
  volume = {372},
  pages = {972-995},
  year = {2018},
  issn = {0021-9991},
  doi = {https://doi.org/10.1016/j.jcp.2017.10.026},
  url = {https://www.sciencedirect.com/science/article/pii/S0021999117307799},
  author = {A. Main and G. Scovazzi}
}

@article{main2018shiftedVol2,
  title = {{The shifted boundary method for embedded domain computations. Part II: Linear advection--diffusion and incompressible Navier--Stokes equations}},
  author = {Main, A and Scovazzi, G},
  journal = {Journal of Computational Physics},
  volume = {372},
  pages = {996--1026},
  year = {2018},
  publisher = {Elsevier}
}

@article{BlockPreconditioners,
  author = {Krzyzanowski, Piotr},
  title = {On Block Preconditioners for Nonsymmetric Saddle Point Problems},
  journal = {SIAM Journal on Scientific Computing},
  volume = {23},
  number = {1},
  pages = {157-169},
  year = {2001},
  doi = {10.1137/S1064827599360406},
  url = {https://doi.org/10.1137/S1064827599360406},
  eprint = {https://doi.org/10.1137/S1064827599360406}
}

@article{dealii2024,
  title = {Thedeal. II library, version 9.7},
  author = {Arndt, Daniel and Bangerth, Wolfgang and Bergbauer, Maximilian and Blais, Bruno and Fehling, Marc and Gassm{\"o}ller, Rene and Heister, Timo and Heltai, Luca and Kronbichler, Martin and Maier, Matthias and others},
  journal = {Journal of Numerical Mathematics},
  year = {2025},
  publisher = {De Gruyter}
}

@article{atallah2020analysis,
  title = {Analysis of the shifted boundary method for the Stokes problem},
  journal = {Computer Methods in Applied Mechanics and Engineering},
  volume = {358},
  pages = {112609},
  year = {2020},
  issn = {0045-7825},
  doi = {https://doi.org/10.1016/j.cma.2019.112609},
  url = {https://www.sciencedirect.com/science/article/pii/S0045782519304852},
  author = {Nabil M. Atallah and Claudio Canuto and Guglielmo Scovazzi}
}

@article{wichrowski2025geometric,
  title = {A Geometric Multigrid Preconditioner for Discontinuous Galerkin Shifted Boundary Method},
  author = {Wichrowski, Michal},
  journal = {arXiv preprint arXiv:2506.12899},
  year = {2025}
}

@article{shypatch,
  title = {A Geometric Multigrid Preconditioner for Shifted Boundary Method},
  author = {Michał Wichrowski and Ajay Ajith},
  year = {2026},
  eprint = {2601.10399},
  archiveprefix = {arXiv},
  primaryclass = {math.NA},
  url = {https://arxiv.org/abs/2601.10399},
  note = {Accepted into CMAME}
}

@article{BlockPrec2,
  author = {Krzyżanowski, Piotr},
  title = {On block preconditioners for saddle point problems with singular or indefinite (1, 1) block},
  journal = {Numerical Linear Algebra with Applications},
  volume = {18},
  number = {1},
  pages = {123-140},
  doi = {https://doi.org/10.1002/nla.717},
  url = {https://onlinelibrary.wiley.com/doi/abs/10.1002/nla.717},
  eprint = {https://onlinelibrary.wiley.com/doi/pdf/10.1002/nla.717},
  year = {2011}
}

@article{baumann1999discontinuous,
  title = {{A discontinuous hp finite element method for convection—diffusion problems}},
  author = {Baumann, Carlos Erik and Oden, J Tinsley},
  journal = {Computer Methods in Applied Mechanics and Engineering},
  volume = {175},
  number = {3-4},
  pages = {311--341},
  year = {1999},
  publisher = {Elsevier}
}

@article{PreconditioningPDEs,
  author = {Mardal, Kent-Andre and Winther, Ragnar},
  title = {Preconditioning discretizations of systems of partial differential equations},
  journal = {Numerical Linear Algebra with Applications},
  volume = {18},
  number = {1},
  pages = {1-40},
  doi = {https://doi.org/10.1002/nla.716},
  url = {https://onlinelibrary.wiley.com/doi/abs/10.1002/nla.716},
  eprint = {https://onlinelibrary.wiley.com/doi/pdf/10.1002/nla.716},
  year = {2011}
}

@book{Elman2014-gj,
  title = {Finite elements and fast iterative solvers: With applications in
           incompressible fluid dynamics},
  author = {Elman, Howard C and Silvester, David J and Wathen, Andy},
  publisher = {Oxford University Press},
  edition = 2,
  year = 2014,
  address = {London, England}
}

@article{zorrilla2024shifted,
  title = {{A shifted boundary method based on extension operators}},
  author = {Zorrilla, Rub{\'e}n and Rossi, Riccardo and Scovazzi, Guglielmo and Canuto, Claudio and Rodr{\'\i}guez-Ferran, Antonio},
  journal = {Computer Methods in Applied Mechanics and Engineering},
  volume = {421},
  pages = {116782},
  year = {2024},
  publisher = {Elsevier}
}

@article{nitsche1971variationsprinzip,
  title = {{Über ein Variationsprinzip zur Lösung von Dirichlet-Problemen bei Verwendung von Teilräumen, die keinen Randbedingungen unterworfen sind}},
  author = {Nitsche, Joachim A},
  journal = {Abhandlungen aus dem Mathematischen Seminar der Universität Hamburg},
  volume = {36},
  pages = {9--15},
  year = {1971},
  publisher = {Springer}
}

@article{yang2024optimal,
  title = {{Optimal surrogate boundary selection and scalability studies for the shifted boundary method on octree meshes}},
  author = {Yang, Cheng-Hau and Saurabh, Kumar and Scovazzi, Guglielmo and Canuto, Claudio and Krishnamurthy, Adarsh and Ganapathysubramanian, Baskar},
  journal = {Computer Methods in Applied Mechanics and Engineering},
  volume = {419},
  pages = {116686},
  year = {2024},
  publisher = {Elsevier}
}

@article{atallah2022high,
  title = {The high-order Shifted Boundary Method and its analysis},
  journal = {Computer Methods in Applied Mechanics and Engineering},
  volume = {394},
  pages = {114885},
  year = {2022},
  issn = {0045-7825},
  doi = {https://doi.org/10.1016/j.cma.2022.114885},
  url = {https://www.sciencedirect.com/science/article/pii/S0045782522001797},
  author = {Nabil M. Atallah and Claudio Canuto and Guglielmo Scovazzi}
}

@article{atallah2021analysis,
  title = {{Analysis of the shifted boundary method for the Poisson problem in domains with corners}},
  author = {Atallah, N. and Canuto, C. and Scovazzi, G.},
  journal = {Mathematics of Computation},
  volume = {90},
  number = {331},
  pages = {2041--2069},
  year = {2021}
}

@article{atallah2021shifted,
  title = {{The shifted boundary method for solid mechanics}},
  author = {Atallah, Nabil and Canuto, Claudio and Scovazzi, Guglielmo},
  journal = {International Journal for Numerical Methods in Engineering},
  volume = {122},
  number = {20},
  pages = {5935--5970},
  year = {2021},
  publisher = {Wiley Online Library}
}

@article{atallah2024nonlinear,
  title = {{Nonlinear elasticity with the shifted boundary method}},
  author = {Atallah, Nabil and Scovazzi, Guglielmo},
  journal = {Computer Methods in Applied Mechanics and Engineering},
  volume = {426},
  pages = {116988},
  year = {2024},
  publisher = {Elsevier}
}

@article{li2020shifted,
  title = {{The shifted interface method: a flexible approach to embedded interface computations}},
  author = {Li, K. and  Atallah, N. and Main, A. and Scovazzi, G.},
  journal = {International Journal for Numerical Methods in Engineering},
  volume = {121},
  number = {3},
  pages = {492--518},
  year = {2020},
  publisher = {Wiley Online Library}
}

@article{collins2023penalty,
  title = {{A penalty-free shifted boundary method of arbitrary order}},
  author = {Collins, J Haydel and Lozinski, Alexei and Scovazzi, Guglielmo},
  journal = {Computer Methods in Applied Mechanics and Engineering},
  volume = {417},
  pages = {116301},
  year = {2023},
  publisher = {Elsevier}
}

@article{kuzmin2022unfitted,
  title = {{An unfitted finite element method using level set functions for extrapolation into deformable diffuse interfaces}},
  author = {Kuzmin, Dmitri and B{\"a}cker, Jan-Phillip},
  journal = {Journal of Computational Physics},
  volume = {461},
  pages = {111218},
  year = {2022},
  publisher = {Elsevier}
}

@article{xue2021new,
  title = {{A new finite element level set reinitialization method based on the shifted boundary method}},
  author = {Xue, Tianju and Sun, WaiChing and Adriaenssens, Sigrid and Wei, Yujie and Liu, Chuanqi},
  journal = {Journal of Computational Physics},
  volume = {438},
  pages = {110360},
  year = {2021},
  publisher = {Elsevier}
}

@article{xu2024weighted,
  title = {{A weighted shifted boundary method for immersed moving boundary simulations of Stokes' flow}},
  author = {Xu, D. and Colom{\'e}s, O. and Main, A. and Li, K. and Atallah, N. and Abboud, N. and Scovazzi, G.},
  journal = {Journal of Computational Physics},
  volume = {510},
  pages = {113095},
  year = {2024},
  publisher = {Elsevier}
}

@article{deprenter2019preconditioning,
  author = {de Prenter, F. and Verhoosel, C. V. and van Brummelen, E. H.},
  title = {Preconditioning immersed isogeometric finite element methods with
           application to flow problems},
  journal = {Computer Methods in Applied Mechanics and Engineering},
  year = {2019},
  volume = {348},
  pages = {604--631}
}

@article{gross2023cutfem,
  author = {Gross, Sven and Reusken, Arnold},
  title = {Analysis of optimal preconditioners for {CutFEM}},
  journal = {Numerical Linear Algebra with Applications},
  year = {2023},
  volume = {30},
  pages = {e2486}
}

@article{zhang2014ibm,
  author = {Zhang, Qinghai and Guy, Robert D. and Philip, Bobby},
  title = {A projection preconditioner for solving the implicit immersed
           boundary equations},
  journal = {Numerical Mathematics: Theory, Methods and Applications},
  year = {2014},
  volume = {7},
  pages = {473--498}
}

@misc{bano2025ibm,
  author = {Ba{\~n}o, I. and others},
  title = {A {SIMPLE}-Based Preconditioned Solver for the Direct-Forcing
           Immersed Boundary Method},
  year = {2025},
  note = {arXiv:2501.15314}
}

@article{wichrowski2025matrix,
  author = {Wichrowski, Micha{\l}},
  title = {Multigrid p-Robustness at {Jacobi} Speeds: Efficient Matrix-Free
           Implementation of Local p-Multigrid Solvers},
  year = {2025},
  journal = {arXiv preprint arXiv:2512.02577}
}

@article{wichrowski2025local,
  title = {Local solvers for high-order patch smoothers via p-multigrid},
  author = {Wichrowski, Micha{\l}},
  journal = {arXiv preprint arXiv:2510.17785},
  year = {2025}
}

@article{starke1997fov,
  title={Field-of-values analysis of preconditioned iterative methods for nonsymmetric elliptic problems},
  author={Starke, Gerhard},
  journal={Numerische Mathematik},
  volume={78},
  number={1},
  pages={103--117},
  year={1997},
  publisher={Springer}
}

@article{klawonn1999fov,
  title={Block triangular preconditioners for nonsymmetric saddle point problems: field-of-values analysis},
  author={Klawonn, Axel and Starke, Gerhard},
  journal={Numerische Mathematik},
  volume={81},
  number={4},
  pages={577--594},
  year={1999},
  publisher={Springer}
}

@article{loghin2004analysis,
  title={Analysis of preconditioners for saddle-point problems},
  author={Loghin, Daniel and Wathen, Andrew J},
  journal={SIAM Journal on Scientific Computing},
  volume={25},
  number={6},
  pages={2029--2049},
  year={2004},
  publisher={SIAM}
}

@article{murphy2000note,
  title={A note on preconditioning for indefinite linear systems},
  author={Murphy, Malcolm F and Golub, Gene H and Wathen, Andrew J},
  journal={SIAM Journal on Scientific Computing},
  volume={21},
  number={6},
  pages={1969--1972},
  year={2000},
  publisher={SIAM}
}

@article{eisenstat1983variational,
  title={Variational iterative methods for nonsymmetric systems of linear equations},
  author={Eisenstat, Stanley C and Elman, Howard C and Schultz, Martin H},
  journal={SIAM Journal on Numerical Analysis},
  volume={20},
  number={2},
  pages={345--357},
  year={1983},
  publisher={SIAM}
}

@article{burman2010ghost,
  title={{Ghost penalty}},
  author={Burman, Erik},
  journal={Comptes Rendus. Math{\'e}matique},
  volume={348},
  number={21-22},
  pages={1217--1220},
  year={2010}
}

@article{burman2015cutfem,
  title={{CutFEM: discretizing geometry and partial differential equations}},
  author={Burman, Erik and Claus, Susanne and Hansbo, Peter and Larson, Mats G and Massing, Andr{\'e}},
  journal={International Journal for Numerical Methods in Engineering},
  volume={104},
  number={7},
  pages={472--501},
  year={2015},
  publisher={Wiley Online Library}
}

@article{main2018shifted,
  title={{The shifted boundary method for embedded domain computations. Part I: Poisson and Stokes problems}},
  author={Main, A. and Scovazzi, G.},
  journal={Journal of Computational Physics},
  volume={372},
  pages={972--995},
  year={2018},
  publisher={Elsevier}
}

@article{atallah2020second,
  title={{The second-generation shifted boundary method and its numerical analysis}},
  author={Atallah, Nabil and Canuto, C. and Scovazzi, G.},
  journal={Computer Methods in Applied Mechanics and Engineering},
  volume={372},
  pages={113341},
  year={2020},
  publisher={Elsevier}
}

@Article{dealii2019design,
  title   = {{The {deal.II} finite element library: Design, features, and insights}},
  author  = {Daniel Arndt and Wolfgang Bangerth and Denis Davydov and
             Timo Heister and Luca Heltai and Martin Kronbichler and
             Matthias Maier and Jean-Paul Pelteret and Bruno Turcksin and
             David Wells},
  journal = {Computers \& Mathematics with Applications},
  year    = {2021},
  DOI     = {10.1016/j.camwa.2020.02.022},
  pages   = {407-422},
  volume  = {81},
  issn    = {0898-1221},
  url     = {https://arxiv.org/abs/1910.13247}
}

@article{hackbusch1985multi,
  title={{The Multi-Grid Method of the Second Kind}},
  author={Hackbusch, Wolfgang},
  journal={Multi-Grid Methods and Applications},
  pages={305--353},
  year={1985},
  publisher={Springer}
}

@article{ANTONELLI2024shiftedIGA,
  title={{The Shifted Boundary Method in Isogeometric Analysis}},
  author={Antonelli, Nicol{\`o} and Aristio, Ricky and Gorgi, Andrea and Zorrilla, Rub{\'e}n and Rossi, Riccardo and Scovazzi, Guglielmo and W{\"u}chner, Roland},
  journal={Computer Methods in Applied Mechanics and Engineering},
  volume={430},
  pages={117228},
  year={2024},
  publisher={Elsevier},
  doi={10.1016/j.cma.2024.117228}
}

@article{wichrowski2022stokes,
  title={A matrix-free multilevel preconditioner for the generalized stokes problem with discontinuous viscosity},
  author={Wichrowski, Micha{\l} and Krzy{\.z}anowski, Piotr},
  journal={Journal of Computational Science},
  volume={63},
  pages={101804},
  year={2022},
  publisher={Elsevier}
}

@article{braess1997efficient,
  title={An efficient smoother for the Stokes problem},
  author={Braess, Dietrich and Sarazin, Regina},
  journal={Applied Numerical Mathematics},
  volume={23},
  number={1},
  pages={3--19},
  year={1997},
  publisher={Elsevier}
}

@article{zulehner2000class,
  title={A class of smoothers for saddle point problems},
  author={Zulehner, Walter},
  journal={Computing},
  volume={65},
  number={3},
  pages={227--246},
  year={2000},
  publisher={Springer}
}

@article{schoberl2003schwarz,
  title={On Schwarz-type smoothers for saddle point problems},
  author={Sch{\"o}berl, Joachim and Zulehner, Walter},
  journal={Numerische Mathematik},
  volume={95},
  number={2},
  pages={377--399},
  year={2003},
  publisher={Springer}
}

@article{arnold2000multigrid,
  title={Multigrid in H (div) and H (curl)},
  author={Arnold, Douglas N and Falk, Richard S and Winther, Ragnar},
  journal={Numerische Mathematik},
  volume={85},
  number={2},
  pages={197--217},
  year={2000},
  publisher={Springer}
}

@article{kanschat2015multigrid,
  title={Multigrid methods for Hdiv-conforming discontinuous Galerkin methods for the Stokes equations},
  author={Kanschat, Guido and Mao, Youli},
  journal={arXiv preprint arXiv:1501.06021},
  year={2015}
}

@article{farrell2019augmented,
  title={An augmented Lagrangian preconditioner for the 3D stationary incompressible Navier--Stokes equations at high Reynolds number},
  author={Farrell, Patrick E and Mitchell, Lawrence and Wechsung, Florian},
  journal={SIAM Journal on Scientific Computing},
  volume={41},
  number={5},
  pages={A3073--A3096},
  year={2019},
  publisher={SIAM}
}

@article{jodlbauer2024matrix,
  title={Matrix-free monolithic multigrid methods for Stokes and generalized Stokes problems},
  author={Jodlbauer, Daniel and Langer, Ulrich and Wick, Thomas and Zulehner, Walter},
  journal={SIAM Journal on Scientific Computing},
  volume={46},
  number={3},
  pages={A1599--A1627},
  year={2024},
  publisher={SIAM}
}

@article{vanka1986block,
  title={Block-implicit multigrid solution of Navier-Stokes equations in primitive variables},
  author={Vanka, S Pratap},
  journal={Journal of Computational Physics},
  volume={65},
  number={1},
  pages={138--158},
  year={1986},
  publisher={Elsevier}
}

@article{Saad1993,
  title   = {A flexible inner-outer preconditioned {GMRES} algorithm},
  author  = {Saad, Yousef},
  journal = {SIAM Journal on Scientific Computing},
  volume  = {14},
  number  = {2},
  pages   = {461--469},
  year    = {1993},
  doi     = {10.1137/0914028}
}
\end{document}